\documentclass[a4paper,11pt]{article}
\usepackage{amsfonts}
\usepackage{amssymb}
\usepackage{amsmath}
\usepackage{amsthm}
\usepackage{bm}
\usepackage{here}
\usepackage{tikz}
\usetikzlibrary{intersections,calc,arrows}
\usepackage{color}
\usepackage{cases}
\numberwithin{equation}{section}
\usepackage{leftidx}
\usepackage{blkarray}

\theoremstyle{definition}
\newtheorem{defn}{Definition}[section]
\newtheorem{rem}[defn]{Remark}
\newtheorem{lemm}[defn]{Lemma}
\newtheorem{prop}[defn]{Proposition}
\newtheorem{thm}[defn]{Theorem}

\newtheorem{exa}[defn]{Example}
\newtheorem{conj}[defn]{Conjecture}

\title{A $3\times3$ linear $q$-difference system with $E_8^{(1)}$-symmetry}
\author{Takahiko Nobukawa
\footnote{Department of Education, Kogakkan University, Japan.
	E-mail: t-nobukawa@kogakkan-u.ac.jp
	\\\ \\
	keywords:
	affine Weyl group, $q$-difference system, $q$-middle convolution.
	MSC2020: 39A13.}
}
\date{}
\allowdisplaybreaks

\begin{document}
	\maketitle
	\begin{center}
		\it
		Dedicated to Professor Yasuhiko Yamada on his 65th birthday
	\end{center}
	\begin{abstract}
		We present a linear $q$-difference equation of rank $3$, which admits the  affine Weyl group symmetry of type $E_8^{(1)}$.
		We further compare this equation with Moriyama--Yamada's quantum curve which has $W(E_8^{(1)})$-symmetry.
		The symmetry of our equation is provided by the $q$-middle convolution, defined by Sakai--Yamaguchi and reformulated by Arai--Takemura.
		In this paper, we provide a reconstruction of the $q$-middle convolution via a $q$-Okubo type equation.
	\end{abstract}
	\section{Introduction}
	Discrete Painlev\'e equations arise as integrable dynamical systems admitting rich symmetry structures governed by affine Weyl groups. 
	Since Sakai's geometric classification \cite{Sakai}, these equations have been understood as discrete analogues of isomonodromic deformations, and they play a central role in the theory of integrable systems.
	The classification is summarized in the following diagram:
	\begin{align*}
		\arraycolsep=1pt
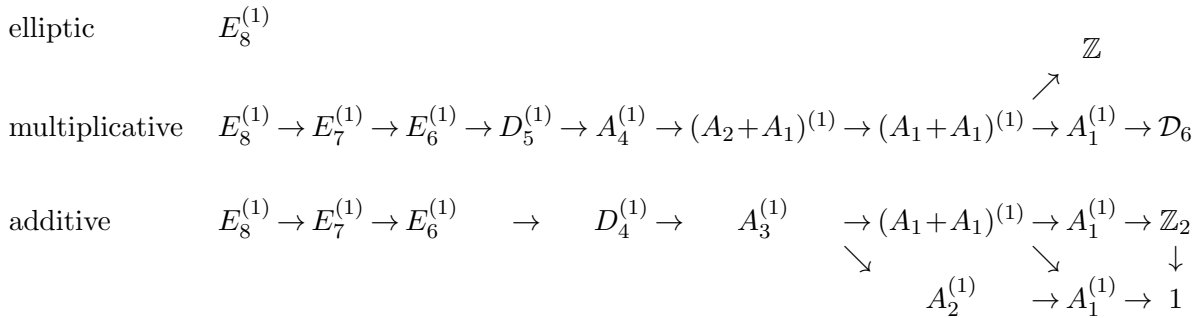
		\begin{array}{lcccccccccccccccccc}
			{\rm elliptic}\quad&E^{(1)}_8\\[-2mm]
			&&&&&&&&&&&&&&&{\mathbb Z}\\
			&&&&&&&&&&&&&&\nearrow\\
			{\rm multiplicative}\quad&E^{(1)}_8&\rightarrow&E^{(1)}_7&\rightarrow&E^{(1)}_6
			&\rightarrow &D^{(1)}_5&\rightarrow&A^{(1)}_4
			&\rightarrow&(A_2\!+\!A_1)^{(1)}&\rightarrow&(A_1\!+\!A_1)^{(1)}
			&\rightarrow&A^{(1)}_1&\rightarrow&{\mathcal D}_6\\[2mm]
			\\
			{\rm additive}\quad&E^{(1)}_8&\rightarrow&E^{(1)}_7&\rightarrow&E^{(1)}_6
			&&\rightarrow&&D^{(1)}_4&\rightarrow&A_3^{(1)}&\rightarrow
			&(A_1\!+\!A_1)^{(1)}&\rightarrow&A^{(1)}_1&\rightarrow&{{\mathbb Z}_2}\\
			&&&&&&&&&&&&\searrow&&\searrow&&&\downarrow\\
			&&&&&&&&&&&&&A^{(1)}_2&\rightarrow&A^{(1)}_1&\rightarrow&1\\
		\end{array}
	\end{align*}
	In Sakai's framework, constructing a birational representation of the affine Weyl group is essential.
	We focus on the $q$-difference case of type $E_8^{(1)}$ in this paper.
	
	In the differential and difference settings, the correspondence between Painlev\'e equations and associated linear problems has been extensively studied (see \cite{KNY} and references therein for instance). 
	In particular, linear differential or difference equations equipped with isomonodromic deformation structures provide a conceptual framework connecting integrable dynamics, representation theory, and mathematical physics. 
	
	The middle convolution, introduced by Katz \cite{Katz} and developed by Dettweiler--Reiter \cite{DR1,DR2}, is a powerful tool to study the theory of linear differential equations.
	The affine Weyl group symmetry of type $D_4^{(1)}$ for the sixth Painlev\'e equation is realized in \cite{Filipuk,Takemura} by applying the middle convolution to the corresponding Lax linear equation (see also \cite{HF}).
	A $q$-analog of the result, i.e. the affine Weyl group symmetry of type $D_5^{(1)}$ via the $q$-middle convolution $mc_\lambda$ \cite{SY}, is studied in \cite{SST}.
	Since $mc_\lambda\circ mc_\mu=mc_{\lambda+\mu}$ holds, and $mc_0=\mathrm{id}$ up to the gauge transform of $GL$, the ($q$-)middle convolution gives a birational transformation for components of the coefficient matrix of equation.
	The ($q$-)middle convolution is useful for giving a birational representation of the affine Weyl group.
	
	A Lax equation for the $q$-Painlev\'e equation of type $E_8^{(1)}$ is provided by Yamada \cite{Yamada}.
	This linear equation is a rank $2$ scalar $q$-difference equation.
	In this paper, we introduce a new rank $3$ linear $q$-difference equation which admits the affine Weyl group symmetry of type $E_8^{(1)}$, given in \eqref{defeq} below.
	The main results are Theorem \ref{E8mc}, which provides a transformation for parameters of the equation \eqref{defeq} by using the $q$-middle convolution.
	We also reconstruct the $q$-middle convolution via a $q$-Okubo type equation.
	Most of formulation is the same as Arai--Takemura's $q$-middle convolution \cite{AT}.
	Only the treatment of the origin, which is important to prove the above main result, is different.
	Properties of the $q$-middle convolution are discussed in Section \ref{secqmc}.
	
	We mention a consideration to get \eqref{defeq} in this paragraph.
	The sixth differential Painlev\'e equation arises as an isomonodromic deformation of a second order Fuchsian differential equation with $4$ singularities on $\mathbb{P}^1=\mathbb{C}\cup\{\infty\}$.
	In the differential case, the position of singularities can be fixed or transformed by M\"obius transformations.
	In other words, every singularities can be treated equally for differential equations.
	The word ``democratic'' is used to express this treatment in \cite{Yoshida}.
	On the other hand, we can not apply general M\"obius transformations to $q$-difference equations because the points $x=0$, $x=\infty$ are fixed points on the $q$-shift operator $T_x:x\mapsto qx$.
	Only the scaling $x\to cx$ and the inversion $x\mapsto x^{-1}$ can be applied to $q$-difference equations.
	Therefore there are several variations for $q$-analog of the above Fuchsian differential equation.
	More precisely, $q$-analogs of second order Fuchsian differential equations with following $4$ singularities are not equivalent:
	\begin{align*}
		&\{t_1,t_2,t_3,t_4\}
		\ \mbox{(``democratic'')},\quad
		\{t_1,t_2,t_3,0\},\quad
		\{t_1,t_2,0,\infty\}.
	\end{align*}
	The $q$-Painlev\'e equation of type $E_7^{(1)}$ arises as a deformation of the first case of the above $q$-difference equation, although the third case corresponds to the $q$-Painlev\'e equation of type $D_5^{(1)}$.
	See \cite{Yamada} for the Lax formalism of $E_7^{(1)}$ case, \cite{Murata,SakaiE6,Yamada} for it of $E_6^{(1)}$ case (the second case of above equations) and \cite{JS,Yamada} for it of $D_5^{(1)}$ case.
	It is found from these results that considering ``democratic'' $q$-difference equation is nice for giving equations with higher symmetry.
	It was found in \cite{Boalch} that the additive Painlev\'e equations of type $E_6^{(1)}$, $E_7^{(1)}$ and $E_8^{(1)}$ arise as Schlesinger transformations of Fuchsian differential equation of rank $3$, $4$ and $6$, respectively.
	The spectral type of these Fuchsian differential equation is given by
	\begin{align}
		E_6^{(1)}:111;111;111,\
		E_7^{(1)}:22;1111;1111,\
		E_8^{(1)}:33;222;111111.
	\end{align}
	Thus it is natural to consider a $q$-analog of the rank $3$ equation with $3$ singularities by obtaining a Lax pair of the $q$-Painlev\'e equation ot type $E_6^{(1)}$.
	Due to the above discussion, a democratic $q$-analog of the rank $3$ Fuchsian differential equation may admit higher symmetry.
	We get the equation \eqref{defeq} from this consideration.
	We remark that the Lax linear equation given by Park \cite{Park} is a $q$-analog the above rank $3$ equation related to $E_6^{(1)}$.
	
	In \cite{MY}, quantum representation of the affine Weyl group of type $E_8^{(1)}$ are mainly discussed by Moriyama--Yamada.
	This representation is obtained by considering birational transformation of a quantum curve (see also \cite{Moriyama} for properties of this curve).
	The quantum $q$-difference Painlev\'e equation is studied as an application.
	In this paper we compare our equation \eqref{defeq} with this curve.
	The reduction of our equation \eqref{defeq} to a scalar equation can be regarded as a non-autonomous version of the curve (see Proposition \ref{propsc}).
	Thus our results would be applied to some quantum integrable systems including the quantum Painlev\'e equation.

	The contents of this paper are as follows.
	In Section \ref{secnot}, we give notation used in this paper.
	In Section \ref{secqmc}, we study the $q$-middle convolution via the $q$-Okubo type equation \eqref{Okuboeq}.
	In Section \ref{secE8}, we introduce a rank $3$ equation and show that this equation has $W(E_8^{(1)})$-symmetry.
	In Section \ref{secsc}, we discuss a relation between our equation and Moriyama-Yamada's quantum curve.
	In Section \ref{secsumm}, we summarize this paper and discuss related topics.
	Proofs of several results are written in Appendices for readability.
	
	\section{Notation}\label{secnot}
	In this section, we give notation used throughout this paper, and show several properties.
	The properties are mainly used  in Section \ref{secqmc}.
	We fix $q\in\mathbb{C}$ with $0<|q|<1$.
	We put
	\begin{align}
		&T_x:f(x)\mapsto f(qx),\\
		&D_x=\frac{1}{(1-q)x}(1-T_x):f(x)\mapsto\frac{f(x)-f(qx)}{x-qx},\\
		&(x)_\infty=\prod_{k=0}^\infty (1-xq^k),\\
		&\int_{[0,\tau]}f(t)d_qt=\int_0^\tau f(t)d_qt=(1-q) \sum_{n=0}^\infty f(\tau q^n)\tau q^n,\\
		&\int_{[0,\tau\infty]} f(t) d_qt=\int_0^{\tau\infty} f(t)d_qt=(1-q) \sum_{n=-\infty}^\infty f(\tau q^n)\tau q^n,\\
		&\int_{[\tau_1,\tau_2]}f(t)d_qt=\int_{\tau_1}^{\tau_2} f(t)d_qt=\int_{[0,\tau_2]}f(t)d_qt-\int_{[0,\tau_1]}f(t)d_qt,
		\notag\\
		&\qquad\qquad\qquad\qquad\qquad\qquad\qquad
		(\tau_1,\tau_2\in\mathbb{C}\cup\mathbb{C}\infty)\\
		&\int_{a_1C_1+\cdots+a_mC_m}f(t)d_qt=a_1\int_{C_1}f(t)d_qt+\cdots+a_m\int_{C_m}f(t)d_qt,
	\end{align}
	where each $C_i$ is in the form of an interval $[\tau_1,\tau_2]$.
	In this paper we call $C$ a path (of Jackson integral) if $C$ is a linear combination of intervals.
	
	The following properties are frequently used in this paper.
	\begin{lemm}
		The $q$-Leibniz rule holds:
		\begin{align}\label{qLeib}
			D_x(f(x)g(x))=f(x)D_x(g(x))+D_x(f(x))g(qx)=f(qx)D_x(g(x))+D_x(f(x))g(x).
		\end{align}
	\end{lemm}
	\begin{proof}
		This can be verified easily by direct calculation.
	\end{proof}
	\begin{lemm}
		For a path $C$, we have
		\begin{align}
			\label{JacksonD}&\int_C D_t f(t)d qt=[f(t)]_{t\in \partial C},\\
			\label{JacksonT}&\int_C T_tf(t)d_qt=q^{-1}\int_C f(t)d_qt-\frac{1-q}{q}[tf(t)]_{t\in\partial C},
		\end{align}
		if the integral converges, where $\displaystyle f(\tau\infty)=\lim_{n\to-\infty}f(\tau q^n)$ and
		\begin{align}
			&[f(t)]_{t\in \partial [\tau_1,\tau_2]}=f(\tau_2)-f(\tau_1),\\
			&[f(t)]_{t\in \partial (a_1C_1+\cdots+a_mC_m)}=a_1[f(t)]_{t\in\partial C_1}+\cdots +a_m[f(t)]_{t\in \partial C_m}.
		\end{align}
	\end{lemm}
	\begin{proof}
		First we prove \eqref{JacksonD}.
		It is sufficient to prove
		\begin{align}\label{proofJacksonD}
			\int_{[\tau_1,\tau_2]}D_tf(t)d_qt=f(\tau_2)-f(\tau_1),
		\end{align}
		for $\tau_1,\tau_2\in\mathbb{C}\cup\mathbb{C}\infty$.
		For $\tau\in\mathbb{C}$, we have
		\begin{align}
			\int_0^\tau D_tf(t)d_qt&=\sum_{n=0}^\infty (f(\tau q^n)-f(\tau q^{n+1}))
			\notag\\
			&=\lim_{N\to\infty}\sum_{n=0}^N (f(\tau q^n)-f(\tau q^{n+1}))
			\notag\\
			&=\lim_{N\to\infty} f(\tau)-f(\tau q^{N+1})
			=f(\tau)-f(0)=[f(t)]_{t\in\partial [0,\tau]},
		\end{align}
		and
		\begin{align}
			\int_0^{\tau\infty} D_tf(t)d_qt&=\sum_{n=-\infty}^\infty (f(\tau q^n)-f(\tau q^{n+1}))
			\notag\\
			&=\lim_{M,N\to\infty}\sum_{n=-M}^N (f(\tau q^n)-f(\tau q^{n+1}))
			\notag\\
			&=\lim_{M,N\to\infty} f(\tau q^{-M})-f(\tau q^{N+1})
			=f(\tau\infty)-f(0)=[f(t)]_{t\in\partial [0,\tau\infty]}.
		\end{align}
		Therefore the equation \eqref{proofJacksonD} holds, in particular the proof of \eqref{JacksonD} is completed.
		Due to the $q$-Leibniz rule, we have
		\begin{align}
			D_t(tf(t))=qtD_tf(t)+f(t)=\frac{q}{1-q}(1-T_t)f(t)+f(t)=-\frac{q}{1-q}T_t f(t)+\frac{1}{1-q}f(t).
		\end{align}
		Integrating both sides, we have
		\begin{align}
			[tf(t)]_{t\in\partial C}=-\frac{q}{1-q}\int_C T_tf(t)d_qt+\frac{1}{1-q}\int_Cf(t)d_qt.
		\end{align}
		This completes the proof of \eqref{JacksonT}.
	\end{proof}
	\begin{rem}
		Some relations between the Jackson integral and the $q$-shift operator are written in  \cite{Kac}.
		For instance, the equation \eqref{JacksonD} is discussed in \cite[Corollary 20.1]{Kac}.
	\end{rem}
	
	\section{$q$-middle convolution}\label{secqmc}
	In this section, we reconstruct the $q$-middle convolution via a $q$-analog of Okubo type system.
	Most of the formulation coincides with that of Arai--Takemura \cite{AT}, which is a reformulate version of Sakai--Yamaguchi's $q$-middle convolution \cite{SY}.
	Our construction clarifies a condition of the integral path of the Euler type Jackson integral transformation.
	
	The middle convolution \cite{DR1,DR2,Katz} for Fuchsian differential equations is constructed by three steps:
	\begin{itemize}
		\item[1.] rewrite the equation as an Okubo type system $(xI-S)\dfrac{dy}{dx}=Ay$, where $S$ is a diagonal matrix and $A$ is a constant matrix,
		\item[2.] apply the Euler type integral transformation $y\mapsto\int y(t)(x-t)^\lambda dt$,
		\item[3.] project the system to a quotient space.
	\end{itemize}
	In the following we construct the $q$-middle convolution following this scheme.

	\begin{defn}
		A $q$-difference system is called $q$-Okubo type if it is given in the following form:
		\begin{align}\label{Okuboeq}
			(xI-S)D_xy=Ay,
		\end{align}
		where $I$ is the identity matrix, $S$ is a diagonal matrix and $A$ is a constant matrix.
	\end{defn}
	\begin{defn}
		We define the $q$-Euler type integral transformation $E_\lambda=E_{\lambda;C}$ as follows:
		\begin{align}
			E_{\lambda;C}[f](x)=x^\lambda \int_C f(t)\frac{(q^{1-\lambda}t/x)_\infty}{(qt/x)_\infty}d_qt.
		\end{align}
	\end{defn}
	\begin{rem}
		Due to the $q$-binomial theorem, we have
		\begin{align}
			x^\lambda \frac{(q^{1-\lambda}t/x)_\infty}{(qt/x)_\infty}\xrightarrow{q\to1}x^\lambda (1-t/x)^\lambda=(x-t)^\lambda.
		\end{align}
		Therefore the transformation $E_{\lambda;C}$ is a $q$-analog of Euler type integral transformation
		\begin{align}
			f(t)\mapsto \int_C f(t)(x-t)^\lambda dt.
		\end{align}
	\end{rem}
	We show some properties of the $q$-Euler type integral transformation in the following.
	First we derive two commuting relations between $E_\lambda$ and difference operators.
	\begin{lemm}
		We have
		\begin{align}
			\label{EulerT}&E_{\lambda;C}[T_x f](x)=q^{-\lambda-1}T_xE_{\lambda;C} [f](x)-x^\lambda \frac{1-q}{q}\left[t\frac{1-q^{-\lambda}t/x}{1-t/x}f(t)\frac{(q^{1-\lambda}t/x)_\infty}{(qt/x)_\infty}\right]_{t\in\partial C},\\
			\label{EulerD}&E_{\lambda;C}[D_xf](x)=q^{-\lambda} D_x E_{\lambda;C} [f](x)+x^\lambda \left[\frac{1-q^{-\lambda}t/x}{1-t/x}f(t)\frac{(q^{1-\lambda}t/x)_\infty}{(qt/x)_\infty}\right]_{t\in\partial C}.
		\end{align}
	\end{lemm}
	\begin{proof}
		%It is sufficient to prove those in the case of $C=[\tau_1,\tau_2]$.
		First we show \eqref{EulerT}.
		We have
		\begin{align}
			\notag&\int_C T_tf(t)\cdot \frac{(q^{1-\lambda}t/x)_\infty}{(qt/x)_\infty}d_qt\\
			\notag=&\int_C T_tT_x\left[f(t) \frac{(q^{1-\lambda}t/x)_\infty}{(qt/x)_\infty}\right]d_qt\\
			\notag=&q^{-1}T_x\int_C f(t)\frac{(q^{1-\lambda}t/x)_\infty}{(qt/x)_\infty}d_qt-\frac{1-q}{q}T_x\left[tf(t)\frac{(q^{1-\lambda}t/x)_\infty}{(qt/x)_\infty}\right]_{t\in\partial C}\\
			&=q^{-1}T_x\int_C f(t)\frac{(q^{1-\lambda}t/x)_\infty}{(qt/x)_\infty}d_qt-\frac{1-q}{q}\left[t\frac{1-q^{-\lambda}t/x}{1-t/x}f(t)\frac{(q^{1-\lambda}t/x)_\infty}{(qt/x)_\infty}\right]_{t\in\partial C}
		\end{align}
		Multiplying both sides by $x^\lambda$, we get \eqref{EulerD}.
		
		Next, we show \eqref{EulerD}.
		Due to the $q$-Leibniz rule, we have
		\begin{align}
			\notag&D_tf(t)\cdot \frac{(q^{1-\lambda}t/x)_\infty}{(qt/x)_\infty}\\
			=&D_t\left[f(t)\frac{(q^{1-\lambda}t/x)_\infty}{(qt/x)_\infty}\right]-T_tf(t)\cdot D_t\frac{(q^{1-\lambda}t/x)_\infty}{(qt/x)_\infty}.
		\end{align}
		By a simple calculation we obtain
		\begin{align}
			D_t\frac{(q^{1-\lambda}t/x)_\infty}{(qt/x)_\infty}=\frac{q}{(1-q)x}(1-q^{-\lambda} T_x^{-1})\frac{(q^{1-\lambda}t/x)_\infty}{(qt/x)_\infty}.
		\end{align}
		Therefore we have 
		\begin{align}
			&\int_C D_tf(t)\cdot \frac{(q^{1-\lambda}t/x)_\infty}{(qt/x)_\infty}d_qt
			\notag\\
			=&\left[f(t)\frac{(q^{1-\lambda}t/x)_\infty}{(qt/x)_\infty}\right]_{t\in\partial C}+\frac{q}{(1-q)x}(q^{-\lambda} T_x^{-1}-1)\int_C T_tf(t)\cdot\frac{(q^{1-\lambda}t/x)_\infty}{(qt/x)_\infty}d_qt.
		\end{align}
		Multiplying both sides by $x^\lambda$ and using \eqref{EulerT}, we finally get \eqref{EulerD}.
	\end{proof}
	\begin{rem}
		The equation \eqref{EulerD} is obviously a $q$-analog of 
		\begin{align}
			\int_C \frac{d}{dt}f(t)\cdot (x-t)^\lambda dt=\frac{d}{dx}\int_C f(t)(x-t)^\lambda dt,
		\end{align}
		if the boundary term vanishes.
		By the limit
		\begin{align}
			\frac{1-q^\mu T_x}{1-q}f(x)\xrightarrow{q\to1}(\vartheta_x+\mu) f(x),
		\end{align}
		where $\vartheta_x=x\dfrac{d}{dx}$, the equation \eqref{EulerT} is regarded as a $q$-analog of 
		\begin{align}
			\int_C \vartheta_t f(t)\cdot (x-t)^\lambda dt=(\vartheta_x-\lambda-1)\int_Cf(t)(x-t)^\lambda dt.
		\end{align}
	\end{rem}
	In the following, we consider a transformation of the equation
	\begin{align}\label{qFuchs}
		D_xy=\sum_{i=1}^N\frac{B_i}{x-t_i}y\quad (B_i:M\times M\mbox{ constant matrix}).
	\end{align}
	\begin{lemm}\label{lemOkubo}
		Suppose that $y$ satisfies the equation \eqref{qFuchs}.
		We put
		\begin{align}\label{OkuboY}
			Y_0=\leftidx{^\mathrm{T}}{\left[\frac{y}{x-t_1},\frac{y}{x-t_2},\ldots,\frac{y}{x-t_N}\right]},
		\end{align}
		where ${}^\mathrm{T}z$ is the transposed matrix of $z$.
		The function $Y_0$ satisfies the $q$-Okubo type system
		\begin{align}\label{makeqOkubo}
			\left(xI-\frac{S}{q}\right)D_xY_0=\frac{1}{q}(B-I)Y_0,
		\end{align}
		where $S$ and $B$ are $MN\times MN$ matrices given by
		\begin{align}\label{defSB}
			S=\begin{bmatrix}
				t_1 I_M&O&\cdots &O\\
				O&t_2 I_M&\ddots&\vdots\\
				\vdots&\ddots &\ddots&O\\
				O&\cdots&O&t_N I_M
			\end{bmatrix},\quad
			B=\begin{bmatrix}
				B_1&B_2&\cdots &B_N\\
				B_1&B_2&\cdots &B_N\\
				\vdots &\vdots & &\vdots\\
				B_1&B_2&\cdots &B_N
			\end{bmatrix}.
		\end{align}
	\end{lemm}
	\begin{proof}
		By the $q$-Leibniz rule \eqref{qLeib}, we have
		\begin{align}
			D_x \frac{y}{x-t_i}&=\frac{1}{qx-t_i}D_xy+D_x\frac{1}{x-t_i}\cdot y
			\notag\\
			&=\frac{1}{qx-t_i}\sum_{j=1}^N\left(\frac{B_j-\delta_{i,j}I}{x-t_j}\right)y,
		\end{align}
		where $\delta_{i,j}$ is the Kronecker delta.
		Therefore we obtain
		\begin{align}
			(qx-t_i)D_x \frac{y}{x-t_i}=\sum_{j=1}^N({B_j-\delta_{i,j}I})\frac{y}{x-t_j}.
		\end{align}
		This completes the proof of \eqref{makeqOkubo}.
	\end{proof}
	\begin{prop}
		Let $y$ be a solution of the equation \eqref{qFuchs} and $Y_0$ is given in \eqref{OkuboY}.
		We suppose that the path $C$ satisfies
		\begin{align}\label{boundary}
			\left[\frac{1-q^{-\lambda}t/x}{1-t/x}(tI-S)Y_0(t) \frac{(q^{1-\lambda}t/x)_\infty}{(qt/x)_\infty}\right]_{t\in\partial C}=0.
		\end{align}
		Then the function $Y=E_{\lambda;C}[Y_0](x)$ satisfies the following system:
		\begin{align}\label{EulerOkubo}
			(xI-S)D_xY=\left(q^\lambda B+[\lambda]I\right)Y.
		\end{align}
		Here $S$ and $B$ are given in \eqref{defSB}, and $[\lambda]=\dfrac{1-q^\lambda}{1-q}$.
		In particular $Y$ satisfies
		\begin{align}\label{EulerFuchs}
			D_x Y=\sum_{i=1}^N\frac{G_i}{x-t_i}Y,
		\end{align}
		where %$[\lambda]=\dfrac{1-q^\lambda}{1-q}$ and
		\begin{align}\label{Gi}
			G_i=\begin{blockarray}{cccccccc}
				\begin{block}{[ccccccc]c}
					&&&O\\
					&&&\vdots\\
					&&&O\\
					q^\lambda B_1&\cdots &q^\lambda B_{i-1}&q^\lambda B_i+[\lambda]I&q^\lambda B_{i+1}&\cdots &q^\lambda B_N &\mbox{$i$-th block}\\
					&&&O\\
					&&&\vdots\\
					&&&O\\
				\end{block}
			\end{blockarray}
		\end{align}
	\end{prop}
	\begin{proof}
		It follows from Lemma \ref{lemOkubo} that
		\begin{align}
			\frac{1-T_x}{1-q}Y_0-\frac{S}{q}D_xY_0=\frac{1}{q}(B-I)Y_0.
		\end{align}
		Applying $E_{\lambda;C}$ to both sides, we have
		\begin{align}
			\frac{1-q^{-\lambda-1}T_x}{1-q}Y+-{S}q^{-\lambda-1}D_xY=\frac{1}{q}(B-I)Y.
		\end{align}
		Here the boundary term vanishes due to \eqref{boundary}.
		This is equivalent to \eqref{EulerOkubo}.
		The equation \eqref{EulerFuchs} can be derived by multiplying $(xI-S)^{-1}$ to both sides of \eqref{EulerOkubo} from the left.
	\end{proof}
	\begin{defn}
		The $q$-convolution $c_\lambda$ is defined by the correspondence $(B_1,\ldots,B_N)\mapsto (G_1,\ldots,G_N)$, where $G_i$ is given in \eqref{Gi}.
	\end{defn}
	The following proposition gives invariant spaces of $G_i$-action, where $c_\lambda:(B_1,\ldots,B_N)\mapsto(G_1,\ldots,G_N)$.
	\begin{prop}[{\cite[Proposition 3.9]{AT}}]\label{propinvariant}
		We define 
		\begin{align}
			&\mathcal{K}=\begin{bmatrix}
				\mathrm{Ker} B_1\\
				\vdots\\
				\mathrm{Ker} B_N
			\end{bmatrix}=\mathrm{Ker}B_1\oplus\cdots\oplus\mathrm{Ker}B_N\subset \mathbb{C}^{MN},
			\\
			&\mathcal{L}=\mathrm{Ker} (G_1+\cdots+G_N)\subset \mathbb{C}^{MN}.
		\end{align}
		The spaces $\mathcal{K}$ and $\mathcal{L}$ are invariant under the $G_i$-action for $i=1,\ldots,N$.
	\end{prop}
	\begin{proof}
		This can be checked straightforwardly.
	\end{proof}
	\begin{defn}
		We denote the matrix induced from the $G_i$-action on the quotient space $\mathbb{C}^{MN}/(\mathcal{K}+\mathcal{L})$ by $\bar{G}_i$ $(i=1,\ldots,N)$.
		The $q$-middle convolution $mc_\lambda$ is defined by the correspondence $(B_1,\ldots,B_N)\mapsto (\bar{G}_1,\ldots,\bar{G}_N)$.
	\end{defn}
	\begin{rem}
		This $q$-middle convolution is almost the same as Arai--Takemura's one \cite[Definition 3.1]{AT}.
		Hence, several properties of Arai--Takemura's $q$-middle convolution hold for our $q$-middle convolution (see Proposition \ref{propinvariant}, \ref{proppreserve}, \ref{propmc0}, \ref{propirr} and Theorem \ref{thmcompose}).
		The only difference is the treatment of the point $x=0$ which is important to prove several results in Section \ref{secE8}.
		In Arai--Takemura's construction, the point $x=0$ should be regarded as a singularity even if the corresponding residue matrix is $O$.
	\end{rem}
	\begin{rem}
		The method to construct the $q$-middle convolution via the $q$-Okubo type equation has been developed in private communications with Shunya Adachi.
	\end{rem}
	\begin{defn}[{\cite[Definition 5.1]{AT}}]\label{defisom}
		Let $V$ and $W$ be finite-dimensional vector spaces and let $\mathbf{B}=(B_1,\ldots,B_N)$ (resp. $\mathbf{C}=(C_1,\ldots,C_N)$) be a tuple of endomorphisms of $V$ (resp. $W$).
		\begin{itemize}
			\item[(i)]  $(V,\mathbf{B})$ is isomorphic to $(W,\mathbf{C})$, if there exists an isomorphism $\phi:V \to W$ of the vector spaces such that $\phi\circ B_j=C_j\circ \phi$ for $i=1,\ldots,N$.
			\item[(ii)] $(V,\mathbf{B})$ is irreducible, if the subspace $V'\subset V$ such that $B_jV'\subset V'$ for all $j=1,\ldots,N$ is only $V'=V$ or $V'=\{0\}$.
			\item[(iii)] $(V,\mathbf{B})$ satisfies ($\ast$), if 
			\begin{align}
				\left(\bigcap_{\substack{i=1\\i\neq j}}^N \mathrm{Ker}(B_i)\right)\cap\mathrm{Ker}(B_j+\tau)=\{0\}\quad\mbox{for any $j=1,\ldots,N$ and $\tau\in\mathbb{C}$}.
			\end{align}
			\item[(iv)] $(V,\mathbf{B})$ satisfies ($\ast\ast$), if 
			\begin{align}
				\left(\sum_{\substack{i=1\\i\neq j}}^N \mathrm{Im}(B_i)\right)+\mathrm{Im}(B_j+\tau)=V\quad\mbox{for any $j=1,\ldots,N$ and $\tau\in\mathbb{C}$}.
			\end{align}
		\end{itemize}
	\end{defn}
	\begin{prop}[{\cite[Proposition 5.6]{AT}}]\label{proppreserve}
		The conditions ($\ast$) and ($\ast\ast$) are preserved by the $q$-middle convolution $mc_\lambda$.
	\end{prop}
	\begin{prop}[{\cite[Proposition 5.7]{AT}}]\label{propmc0}
		If $(V,\mathbf{B})$ satisfies ($\ast$), then $(mc_0 (V),mc_0 (\mathbf(B)))$ is isomorphic to $(V,\mathbf{B})$.
	\end{prop}
	\begin{prop}[{\cite[Theorem 5.9]{AT}}]\label{propirr}
		If $(V,\mathbf{B})$ is irreducible, then $(mc_\lambda(V),mc_{\lambda}(\mathbf{B}))$ is irreducible or $V=\{0\}$.
	\end{prop}
	\begin{thm}[{\cite[Theorem 5.8]{AT}}]\label{thmcompose}
		Suppose that $(V,\mathbf{B})$ satisfies the conditions $(\ast)$ and $(\ast\ast)$.
		\begin{itemize}
			\item[(i)] $(mc_{\lambda_2}(mc_{\lambda_1}(V)),mc_{\lambda_2}(mc_{\lambda_1}(\mathbf{B})))$ is isomorphic to $(mc_{\lambda_1+\lambda_2}(V),mc_{\lambda_1+\lambda_2}(\mathbf{B}))$ for any $\lambda_1,\lambda_2\in\mathbb{C}$.
			\item[(ii)] $(mc_{-\lambda}(mc_{\lambda}(V)),mc_{-\lambda}(mc_{\lambda}(\mathbf{B})))$ is isomorphic to $(V,\mathbf{B})$ for any $\lambda\in\mathbb{C}$.
		\end{itemize}
	\end{thm}
	\begin{rem}
		Proposition \ref{propirr} is not used directly in this paper.
	\end{rem}
	
	The $q$-middle convolution is a transformation of a $q$-difference system in the following form:
	\begin{align}
		D_xy=\sum_{i=1}^N \frac{B_i}{x-t_i}\quad(\mbox{$B_i$ : $M\times M$ constant matrix}).
	\end{align}
	In the differential case, there is another transformation, so called addition.
	The addition is a transform associated with $y\mapsto y (x-t_j)^\lambda$.
	We give the addition of $q$-difference systems although this is not used in this paper.
	We consider the gauge transformation $\displaystyle y \mapsto y'=y\frac{(x/t_j')_\infty}{(x/t_j)_\infty}$ $(t_j\neq 0)$ for the above equation.
	Due to simple calculation, we find that $y'$ satisfies the following system:
	\begin{align}\label{formeq}
		D_xy'=\left(\sum_{i\neq j}\frac{B_i'}{x-t_i}+\frac{B_j'}{x-t_j'}\right)y',
	\end{align}
	where
	\begin{align}
		&B_i'=\frac{t_j'}{t_j}\frac{t_i-t_j}{t_i-t_j'}B_i,\\
		&B_j'=\sum_{i\neq j}\frac{t_j'}{t_j}\frac{t_j'-t_j}{t_j'-t_i}B_i+\frac{t_j'}{t_j}B_j+\frac{t_j-t_j'}{t_j(1-q)}I.
	\end{align}
	\begin{defn}[{\cite[section 2]{AT}}]
		We define the addition as
		\begin{align}
			(t_1,\ldots,t_N)\mapsto(t_1,\ldots,t_j',\ldots,t_N)\quad \mbox{and}\quad (B_1,\ldots,B_N)\mapsto (B_1',\ldots,B_N').
		\end{align}
	\end{defn}
	\begin{rem}
		We put $t_j'=q^\alpha t_j$.
		Taking the classical limit $q\to1$, the matrices $B_i'$ ($i\neq i$) and $B_j'$ tend to
		\begin{align}
			B_i'\to B_i,\quad B_j'\to B_j+\alpha I.
		\end{align}
		This transform is the addition on the differential case.
		We remark that the gauge factor $\dfrac{(x/t_j')_\infty}{(x/t_j)_\infty}$ tends to $(1-x/t_j)^\alpha$.
	\end{rem}
	\begin{rem}
		When $t_j=0$, we should consider the gauge transformation $y\mapsto y x^\alpha$ as an analog of the differential case.
		The equation \eqref{formeq} is transformed by this gauge to
		\begin{align}
			B_i\mapsto q^\alpha B_i\ (i\neq j),\quad B_j\mapsto q^\alpha B_j+[\alpha] I.
		\end{align} 
	\end{rem}
	
	\section{A $3\times3$ system}\label{secE8}
	In this section, we introduce a $q$-difference system of rank $3$, which has  the symmetry of the affine Weyl group of type $E_8^{(1)}$.
	We show that the generators of this affine Weyl group are given by permutations of parameters and the $q$-middle convolution.
	
	In this section we consider the following system:
	\begin{align}
		\label{defeq}&T_x y=Ay,\\
		\label{defA}&A=A(x)=I+ A^{(1)}x + A^{(2)}x^2 +\kappa  Ix^3\ : \ 3\times 3\mbox{ matrix},\\
		\label{detA}&\det A=\kappa^3 \prod_{i=1}^9(x+e_i)\quad(e_i\neq e_j\ \mbox{for any $i\neq j$}).
	\end{align}
	The condition
	\begin{align}
		\kappa^3 \prod_{i=1}^9 e_i=1,
	\end{align}
	is needed since $A(0)=I$.
	Due to \eqref{detA}, there are $8$ constraint conditions for components of $A_1$ and $A_2$.
	The equation \eqref{defeq} also has the $GL(3)$ gauge freedom.
	In conclusion the equation \eqref{defeq} essentially has two accessory parameters.
	
	We first give an example of expression of  $A$.
	If matrices $X_1$, $X_2$, $X_3$ satisfy
	\begin{align}
		\label{XXX}&X_1X_2X_3=\kappa I,\\
		\label{detX1}&\det(I+x X_1)=(1+x/e_1)(1+x/e_2)(1+x/e_3),\\
		\label{detX2}&\det(I+x X_2)=(1+x/e_4)(1+x/e_5)(1+x/e_3),\\
		\label{detX3}&\det(I+x X_3)=(1+x/e_1)(1+x/e_2)(1+x/e_3),
	\end{align}
	then $A=(I+X_1)(I+X_2)(I+X_3)$ satisfies the conditions \eqref{defA} and \eqref{detA}.
	We can transform $X_1$ and $X_3$ simultaneously as
	\begin{align}
		\label{XXtrig}X_1=\begin{bmatrix}
			1/e_1&a_1&a_2\\
			0&1/e_2&a_3\\
			0&0&1/e_3
		\end{bmatrix},\ X_3=\begin{bmatrix}
			1/e_7&0&0\\
			a_4&1/e_7&0\\
			a_5&a_6&1/e_9
		\end{bmatrix},
	\end{align}
	by a suitable $GL(3)$ gauge transformation.
	We put 
	\begin{align}\label{X2trig}
		X_2=\kappa X_1^{-1}X_3^{-1}.
	\end{align}
	The condition 
	\begin{align}
		\label{trigdetX2}\det(I+X_2)=(1+x/e_4)(1+x/e_5)(1+x/e_3)
	\end{align} 
	can be solved by putting $a_1$ and $a_2$ as suitable rational functions on parameters and $a_3$, $a_4$, $a_5$, $a_6$.
	The gauge transformation by diagonal matrices preserves the forms of $X_1$ and $X_3$.
	This means that $T_xy=(I+X_1)(I+X_2)(I+X_3)y$ has two dimensional gauge freedom.
	Therefore only two parameters (accessory parameters) are independent from \eqref{XXX}, \eqref{detX1}, \eqref{detX2} and \eqref{detX3}.
	The equation 
	\begin{align}
		\label{trigeq}T_xy=(I+X_1)(I+X_2)(I+X_3)y,
	\end{align} 
	where $X_1$, $X_2$ and $X_3$ are given in \eqref{XXtrig}, \eqref{X2trig} and \eqref{trigdetX2}, is used in Section \ref{secsc}.
	
	We characterize the equation \eqref{defeq}.
	\begin{defn}[{\cite[Definition 3.1, 3.2, 3.3]{SY}}]
		We define a spectral type $\mathcal{S}=(\mathcal{S}_0;\mathcal{S}_\infty;\mathcal{S}_{\mathrm{div}})$ of the equation
		\begin{align}
			T_x{y}=A(x) y=\sum_{k=0}^N A_k x^k y\quad (\mbox{$A_i$: $M\times M$ matrix}),
		\end{align}
		as follows.
		\begin{itemize}
			\item[1.] Let
			\begin{align}
				A_0\sim \bigoplus_{i=1}^l \bigoplus_{j=1}^{s_i}J(\alpha_i,t_{i,j})\quad (\mbox{$J(\alpha,t)$: Jordan cell, }t_{i,j}\geq t_{i,j+1}).
			\end{align}
			be a Jordan decomposition of $A_0$ and $m_i=(m_{i,1},m_{i,2},\dots,m_{i,t_{i,1}})$ be the conjugate Young diagram of $t_i=(t_{i,1},t_{i,2},\dots,t_{i,s_i})$.
			We put
			\begin{align}
				\mathcal{S}_0:m_{1,1}\ldots m_{1,t_{1,1}}, m_{2,1}\ldots m_{2,t_{2,1}},\ldots,m_{l,1}\ldots m_{l,t_{l,1}}.
			\end{align}
			We also define $\mathcal{S}_\infty$ in the same way as replacing $A_0$ by $A_N$.
			\item[2.] We set $Z_A=\{a\in\mathbb{C}\mid \det A(a)=0\}=\{a_1,\ldots,a_l\}$.
			Let $d_k$ $(1\leq k\leq M,\ d_{k+1}|d_k)$ be the elementary divisors of $\det A(x)$.
			For $a_i\in Z_A$, let $\tilde{n}_{k}^i$ be the order of zeros of $a_i$ of $d_k$.
			We set $n^i=(n_1^i,\ldots,n_{\tilde{n}_1^i}^i)$ as the conjugate Young diagram of $\tilde{n}^i=(\tilde{n}_1^i,\ldots,\tilde{n}_{M}^i)$.
			We put
			\begin{align}
				\mathcal{S}_{\mathrm{div}}:n_1^1,\ldots,n_{\tilde{n}_1^1}^1,n_1^2,\ldots,n_{\tilde{n}_1^2}^2,\ldots,n_1^l,\ldots,n_{\tilde{n}_1^l}^l.
			\end{align}
		\end{itemize}
	\end{defn}
	We characterize the equation \eqref{defeq} by its spectral type.
	\begin{prop}
		Suppose that the spectral type of the equation
		\begin{align}
			T_x{y}=A(x) y,
		\end{align}
		is given by 
		\begin{align}\label{spectral}
			(3;3;1,1,1,1,1,1,1,1,1).
		\end{align}
		Then the matrix $A(x)$ satisfies the following conditions:
		\begin{align}
			&A(x)=A^{(0)}+A^{(1)} x+A^{(2)} x^2+A^{(3)} x^3\quad (\mbox{$A^{(i)}$; $3\times 3$ matrix}),\\
			&A^{(0)}=a I,\ A^{(3)}=b I,\\
			&\det A(x)=b^3\prod_{i=1}^9 (x+c_i),
		\end{align}
		where $a$, $b$, $c_i\in\mathbb{C}^\times$ are suitable parameters such that
		\begin{align}
			c_i\neq c_j\ (i\neq j),\quad b^3\prod_{i=1}^9 c_i=a^3.
		\end{align}
		In particular, the equation \eqref{defeq} is characterized by the above spectral type.
	\end{prop}
	\begin{proof}
		By $\mathcal{S}_0$ and $\mathcal{S}_\infty$, we find that the matrix
		\begin{align}
			A(x)=\sum_{k=0}^N A^{(k)} x^k,
		\end{align}
		is a $3\times 3$ matrix and there are $a$, $b\in\mathbb{C}^\times$ such that
		\begin{align}
			A^{(0)}= a I,\ A^{(N)}=bI.
		\end{align}
		Due to $\mathcal{S}_{\mathrm{div}}$, there are distinct $c_1,\ldots,c_9\in\mathbb{C}^\times$ such that
		\begin{align}
			\det A(x)=b^3\prod_{i=1}^9 (x+c_i).
		\end{align} 
		Therefore we find $N=3$.
		This completes the proof of conditions of $A(x)$.
		Applying the gauge transformation by $x^\alpha$ where $q^\alpha=a$, the equation $T_x{y}=A(x) y$ is transformed to
		\begin{align}
			T_x y =\frac{1}{a}A(x) y.
		\end{align}
		This equation is equivalent to \eqref{defeq} with $c_i=e_i$ and $\dfrac{b}{a}=\kappa$.
	\end{proof}
	\begin{rem}
		The spectral type $\mathcal{S}_0=3$ (resp. $\mathcal{S}_\infty=3$) means that the point $x=0$ (resp. $x=\infty$) is regular up to gauge transformations.
		The spectral type $\mathcal{S}_\mathrm{div}=1,1,1,1,1,1,1,1,1$ may mean that there are $3$ singularities on $\mathbb{C}^\times$ for the classical limit $q\to1$ of the equation.
		In particular, the equation \eqref{defeq} can be regarded as a democratic $q$-analog of the rank $3$ Fuchsian differential equation \cite{Boalch} related to the additive Painlev\'e equation of type $E_6^{(1)}$.
	\end{rem}

	It is obvious that the symmetric group $\mathfrak{S}_9$ acts on the equation \eqref{defeq} by
	\begin{align}
		\sigma\in\mathfrak{S}_9:e_i\mapsto e_{\sigma(i)}.
	\end{align}
	The following theorem is crucial to construct the affine Weyl group symmetry of type $E_8^{(1)}$.
	%The parameters $e_1$, $e_2$, $e_3$ correspond to a distinguished triple associated with the choice of gauge.
	\begin{thm}\label{E8mc}
		We put $g=(-x/e_1)_\infty(-x/e_2)_\infty(-x/e_3)_\infty$ and $q^\lambda=(\kappa e_1 e_2 e_3)^{-1}$.
		We apply the following transformations to the equation \eqref{defeq} in order:
		\begin{itemize}
			\item[1.] the gauge transformation by $g$,
			\item[2.] $q$-middle convolution $mc_\lambda$,
			\item[3.] the gauge transformation by $g^{-1}$. 
		\end{itemize}
		Then the transformed equation is given below:
		\begin{align}
			&T_x y=\tilde{A}y,\\
			\label{Apro1}&\tilde{A}=I+ \tilde{A}^1x +  \tilde{A}^2 x^2+ \tilde{\kappa} Ix^3 \ :\ 3\times 3\mbox{ matrix},\\
			\label{Apro2}&\det \tilde{A}=\tilde{\kappa}^3 \prod_{i=1}^9 (x+\tilde{e}_i),
		\end{align}
		where
		\begin{align}
			\tilde{e}_i=\begin{cases}
				e_i&(i=1,2,3)\\
				\kappa e_1e_2e_3 e_i &(i=4,5,6,7,8,9)
			\end{cases},\quad \tilde{\kappa}=\frac{1}{\kappa(e_1e_2e_3)^2}.
		\end{align}
	\end{thm}
	\begin{proof}
		See Appendix \ref{appA}.
	\end{proof}
	According to Proposition \ref{E8mc}, the action
	\begin{align}
		&s_0:[e_1,e_2,e_3,e_4,e_5,e_6,e_7,e_8,e_9,\kappa]\notag\\
		&\mapsto[e_1,e_2,e_3,\kappa e_1e_2e_3 e_4,\kappa e_1e_2e_3e_5,\kappa e_1e_2e_3e_6,\kappa e_1e_2e_3e_7,\kappa e_1e_2e_3e_8,\kappa e_1e_2e_3e_9,\kappa^{-1}(e_1e_2e_3)^{-2}],
	\end{align}
	on parameters can be described by using the $q$-middle convolution.
	We put $s_i: e_i\leftrightarrow e_{i+1}$ $(i=1,\ldots,8)$.
	The equation \eqref{defeq} has the affine Weyl group symmetry of type $E_8^{(1)}$ as actions on parameters:
	\begin{prop}\label{propdynkin}
		The maps $s_0$, $s_1$,$\ldots$, $s_8$ on $\{[e_1,\ldots,e_9,\kappa]\mid \kappa^3 e_1\cdots e_9=1\}$ satisfy the Coxeter relations of type $E_8^{(1)}$.
		The Dynkin diagram is as follows:
		\begin{center}
			\begin{tikzpicture}
				%\draw[help lines] (0,0) grid (10,5);
				\draw (4,2) node {$s_0$};
				\draw (4,1.5)--(4,0.5);
				\draw (0,0) node {$s_1$};
				\draw (0.5,0)--(1.5,0);
				\draw (2,0) node {$s_2$};
				\draw (2.5,0)--(3.5,0);
				\draw (4,0) node {$s_3$};
				\draw (4.5,0)--(5.5,0);
				\draw (6,0) node {$s_4$};
				\draw (6.5,0)--(7.5,0);
				\draw (8,0) node {$s_5$};
				\draw (8.5,0)--(9.5,0);
				\draw (10,0) node {$s_6$};
				\draw (10.5,0)--(11.5,0);
				\draw (12,0) node {$s_7$};
				\draw (12.5,0)--(13.5,0);
				\draw (14,0) node {$s_8$};
			\end{tikzpicture}
		\end{center}
	\end{prop}
	\begin{proof}
		This can be checked directly.
	\end{proof}

	We define $A_i$ as
	\begin{align}\label{defAisec4}
		\sum_{i=1}^3\frac{A_i}{x+e_i}=\frac{I-A'}{(1-q)x},\quad A'=\frac{A}{(1+x/e_1)(1+x/e_2)(1+x/e_3)},
	\end{align}
	and assume that $(A_1,A_2,A_3)$ satisfies the conditions ($\ast$) and ($\ast\ast$) in Definition \ref{defisom} (iii) and (iv) in the following.
	This assumption is used in the proof of Proposition \ref{propbraid} below.
	We give a birational representation of the affine Weyl group $W(E_8^{(1)})$ via transformations of the equation \eqref{defeq}.
	We assume that $\kappa e_ie_je_k\neq 1$, where $i\neq j\neq k\neq i$.
	We use the notation $B\simeq C$ for matrices $B$ and $C$ if there exists $G$ such that $G^{-1}BG=C$.
	We put $s_0.A=\tilde{A}$, where $\tilde{A}$ is given in Theorem \ref{E8mc}.
	In other words, $s_0$ is defined as $s_0=g^{-1}\circ mc_\lambda \circ g$ where $g=(-x/e_1)_\infty(-x/e_2)_\infty(-x/e_3)_\infty$ and $q^\lambda=(\kappa e_1e_2e_3)^{-1}$.
	The following proposition is crucial.
	\begin{prop}\label{propbraid}
		We have
		\begin{align}
			\label{braid0i}&s_0.s_i.A\simeq s_i.s_0.A\quad (i\neq 0,3),\\
			\label{braid00}&s_0.s_0.A\simeq A.
			%\label{braid03}&s_0.s_3.s_0.A\simeq s_3.s_0.s_3.A.
		\end{align}
	\end{prop}
	\begin{proof}
		See Appendix \ref{appB}.
	\end{proof}
	Clearly, the map $s_0$ defines a rational transformation on the components of $A$.
	Due to \eqref{braid00}, we have $s_0^2=\mathrm{id}$ up to conjugation by $GL(3)$.
	In this sense, the map $s_0$ is birational.
	\begin{conj}\label{conjrepE8}
		The set
		\begin{align}
			\{A\mid \mbox{$A$ satisfies \eqref{defA} and \eqref{detA}, $e_1,\ldots,e_9,\kappa\in\mathbb{C}^\times$ such that $\displaystyle\kappa \prod_{i=1}^9 e_i=1$, $e_i\neq e_j$ $(i\neq j)$}\}/\simeq,
		\end{align}
		admits a birational action on the affine Weyl group $W(E_8^{(1)})$.
		Each action and the corresponding Dynkin diagram are as follows:
		\begin{align}
			&s_0=g^{-1}\circ mc_\lambda \circ g\quad \mbox{($g=(-x/e_1)_\infty(-x/e_2)_\infty(-x/e_3)_\infty$ and $q^\lambda=(\kappa e_1e_2e_3)^{-1}$)},\\
			&s_i:e_i\leftrightarrow e_{i+1}\quad (i\neq 3),\\
			&s_3:e_3\leftrightarrow e_4,\quad \mbox{acts on accessory parameters birationally}
		\end{align}
		\begin{center}
			\begin{tikzpicture}
				%\draw[help lines] (0,0) grid (10,5);
				\draw (4,2) node {$s_0$};
				\draw (4,1.5)--(4,0.5);
				\draw (0,0) node {$s_1$};
				\draw (0.5,0)--(1.5,0);
				\draw (2,0) node {$s_2$};
				\draw (2.5,0)--(3.5,0);
				\draw (4,0) node {$s_3$};
				\draw (4.5,0)--(5.5,0);
				\draw (6,0) node {$s_4$};
				\draw (6.5,0)--(7.5,0);
				\draw (8,0) node {$s_5$};
				\draw (8.5,0)--(9.5,0);
				\draw (10,0) node {$s_6$};
				\draw (10.5,0)--(11.5,0);
				\draw (12,0) node {$s_7$};
				\draw (12.5,0)--(13.5,0);
				\draw (14,0) node {$s_8$};
			\end{tikzpicture}
		\end{center}
	\end{conj}
	Proposition \ref{propdynkin} motivates us to find the braid relations of $s_3$ and $s_i$ as operations on the equation \eqref{defeq}.
	The construction of $s_3$ and the corresponding braid relation remain open in general.
	This conjecture is supported by numerical experiments.
	The statement of this conjecture gives a birational representation of $W(E_8^{(1)})$, thus to show this conjecture is important.

	\section{Scalar equation}\label{secsc}
	In this section we transform the equation \eqref{trigeq} to a rank $3$ scalar equation.
	We also compare this scalar equation with Moriyama--Yamada's quantum curve \cite{MY}.
	First we review the point configuration of $q$-difference equation which is introduced in \cite{FN}.
	This is a tool to characterize linear $q$-difference equations.
	\begin{defn}[\cite{FN}]\label{defconfiguration}
		Let $\displaystyle L=\sum_{i=0}^{M}\sum_{j=0}^{N}a_{i,j}x^{i}T_x^j $ be a linear $q$-difference operator where $a_{i,j}\in\mathbb{C}$. 
		\begin{itemize}
			\item[1.] We call the equation $Ly=0$ {Fuchsian $q$-difference equation} if $a_{0,0}$, $a_{0,N}$, $a_{M,0}$, $a_{M,N}\neq0$.\footnote{This condition means $x=0$ and $x=\infty$ are regular singularities.}
			\item[2.] \begin{itemize}
				\item[i.] We rewrite $\displaystyle L=\sum_{i=0}^M x^i L_i(T_x)$.
				The roots of $L_{0}(y)=0$ are called the characteristic roots at $x=0$.
				Similarly the roots of $L_{M}(y)=0$ are called the characteristic roots at $x=\infty$.
				\item[ii.] We rewrite $\displaystyle L=\sum_{j=0}^N P_j(x)T_x^j$.
				The roots of $P_{0}(x)=0$ are called the characteristic roots at $T_x=0$, and the roots of $P_{N}(y)=0$ are called the characteristic roots at $x=\infty$.
			\end{itemize}
			\item[3.] 
			\begin{itemize}
				\item[i.] Suppose that $e, e q,\ldots, e q^{k-1}$ are characteristic roots at $x=0$  (resp. at $T_x=0$).
				We call $\{e, e q,\ldots, e q^{k-1}\}$ ``$k$-th characteristic roots'' if
				\begin{align*}
					&\begin{cases}
						L_0(T_x)\propto(T_x-e)(T_x-e q)\cdots (T_x-e q^{k-2})(T_x-e q^{k-1}),\\
						L_1(T_x)\propto(T_x-e)(T_x-e q)\cdots (T_x-e q^{k-2}),\\
						\quad\vdots\hspace{70pt}\rotatebox{90}{$\ddots$}\\
						L_{k-1}(T_x)\propto(T_x-e)
					\end{cases}\\
					&\left(\mbox{resp. }\begin{cases}
						P_0(x)\propto(x-e)(x-e q)\cdots (x-e q^{k-2})(x-e q^{k-1}),\\
						P_1(x)\propto(x-e)(x-e q)\cdots (x-e q^{k-2}),\\
						\quad\vdots\hspace{70pt}\rotatebox{90}{$\ddots$}\\
						P_{k-1}(x)\propto(x-e)
					\end{cases}\right).
				\end{align*}
				\item[ii.] Suppose that $e, e q^{-1},\ldots, e q^{-(k-1)}$ are characteristic roots at $x=\infty$  (resp. $T_x=\infty$).
				We call $\{e, e q,\ldots, e q^{k-1}\}$ $k$-th roots if
				\begin{align*}
					&\begin{cases}
						L_M(T_x)\propto(T_x-e)(T_x-e q^{-1})\cdots (T_x-e q^{-(k-2)})(T_x-e q^{-(k-1)}),\\
						L_{M-1}(T_x)\propto(T_x-e)(T_x-e q^{-1})\cdots (T_x-e q^{-(k-2)}),\\
						\quad\vdots\hspace{70pt}\rotatebox{90}{$\ddots$}\\
						L_{M-k+1}(T_x)\propto(T_x-e)
					\end{cases}\\
					&\left(\mbox{resp. }\begin{cases}
						P_N(x)\propto(x-e)(x-e q^{-1})\cdots (x-e q^{-(k-2)})(x-e q^{-(k-1)}),\\
						P_{N-1}(x)\propto(x-e)(x-e q^{-1})\cdots (x-e q^{-(k-2)}),\\
						\quad\vdots\hspace{70pt}\rotatebox{90}{$\ddots$}\\
						P_{N-k+1}(x)\propto(x-e)
					\end{cases}\right).
				\end{align*}
			\end{itemize}
			\item[4.] For a Fuchsian $q$-difference equation $Ly=0$, the point configuration of $Ly=0$ is a figure summarizing its characteristic roots and non-logarithmic roots, given by the following rule:
			\begin{itemize}
				\item[i.] Draw 4 lines $x=0$, $x=\infty$, $T_x=0$, $T_x=\infty$.
				\item[ii.] Plot all characteristic roots at $x=0$, $x=\infty$, $T_x=0$, $T_x=\infty$ on the corresponding line. 
				\item [iii.] Use $k$-th point to express $k$-th characteristic roots.
			\end{itemize}
		\end{itemize}
	\end{defn}
	\begin{lemm}[\cite{FN,MY}]
		All characteristic roots of a Fuchsian $q$-difference equation $Ly=0$ is given by
		\begin{itemize}
			\item $a_1,\ldots,a_N$ at $x=0$,
			\item $b_1,\ldots,b_N$ at $x=\infty$,
			\item $c_1,\ldots,c_M$ at $T_x=0$,
			\item $d_1,\ldots,d_M$ at $T_x=\infty$.
		\end{itemize}
		Then we have
		\begin{align}
			a_1\cdots a_Nd_1\cdots d_M=b_1\cdots b_Nc_1\cdots c_M.
		\end{align}
		This is called $q$-Fuchs relation.
	\end{lemm}
	\begin{proof}
		The $q$-Fuchs relation is found by comparing two expression of $L$:
		\begin{align}
			L=\sum_{i=0}^M x^i L_i(T_x)=\sum_{j=0}^N P_j(x)T_x^j.
		\end{align}
	\end{proof}
	\begin{exa}
		Consider a Fuchsian $q$-difference equation $\displaystyle Ly=\sum_{i=0}^6\sum_{j=0}^4 a_{i,j}x^iT_x^j y=0$ having the following roots:
		\begin{itemize}
			\item[1.] $e_1$, $e_2$, $\{e_3,e_3q\}$ (double roots) at $x=0$.
			\item[2.] $e_4$, $e_5$, $e_6$, $e_7$ at $x=\infty$.
			\item [3.] $e_8$, $e_9$, $\{e_{10},e_{10}q\}$ (double roots), $\{e_{11},e_{11}q\}$ (double roots) at $T_x=0$.
			\item[4.] $e_{12}$, $e_{13}$, $e_{14}$, $\{e_{15},e_{15}q^{-1},e_{15}q^{-2}\}$ (triple roots) at $T_x=\infty$.
		\end{itemize}
		Then the corresponding point configuration is given as follows.
		
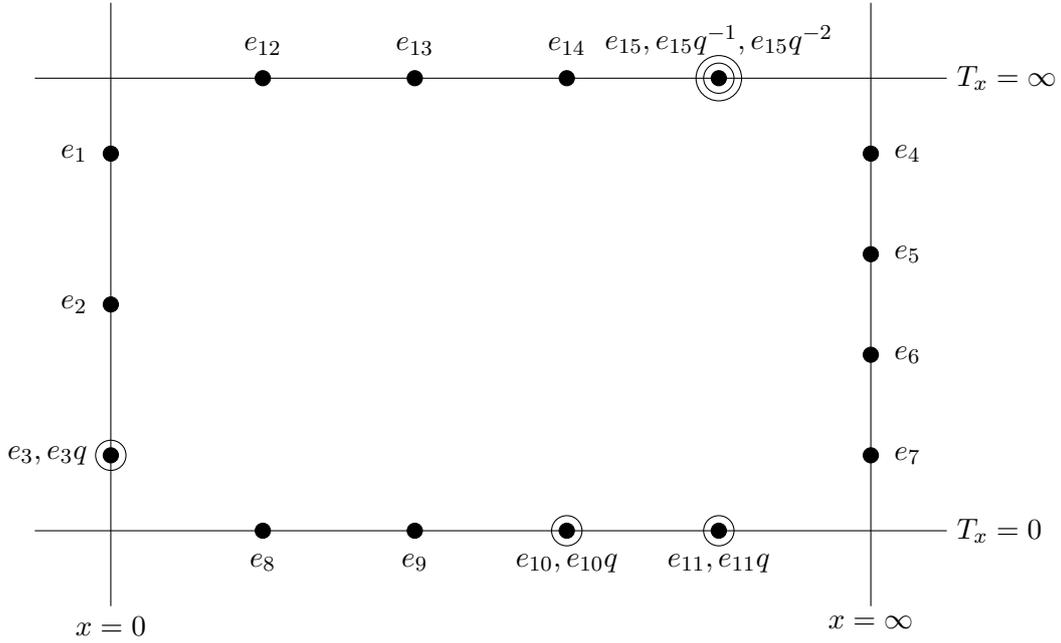
\begin{center}
			\begin{tikzpicture}
				%\draw[help lines] (0,0) grid (13,8);
				\draw (0,1)--(12,1);
				\draw (0,7)--(12,7);
				\draw (1,0)--(1,8);
				\draw (11,0)--(11,8);
				\draw[right] (12,1) node {$T_x=0$};
				\draw[right] (12,7) node {$T_x=\infty$};
				\draw[below] (1,0) node {$x=0$};
				\draw[below] (11,0) node {$x=\infty$};
				\filldraw (1,6) circle [radius=0.1]; 
				\draw[left=5pt] (1,6) node {$e_1$};
				\filldraw (1,4) circle [radius=0.1];
				\draw[left=5pt] (1,4) node {$e_2$};
				\filldraw (1,2) circle [radius=0.1];
				\draw (1,2) circle [radius=0.2];
				\draw[left=5pt] (1,2) node {$e_3,e_3q$};
				\filldraw (11,6) circle [radius=0.1];
				\draw[right=5pt] (11,6) node {$e_4$};
				\filldraw (11,4+2/3) circle [radius=0.1];
				\draw[right=5pt] (11,4+2/3) node {$e_5$};
				\filldraw (11,4-2/3) circle [radius=0.1];
				\draw[right=5pt] (11,4-2/3) node {$e_6$};
				\filldraw (11,2) circle [radius=0.1];
				\draw[right=5pt] (11,2) node {$e_7$};
				\filldraw (3,1) circle [radius=0.1];
				\draw[below=5pt] (3,1) node {$e_8$};
				\filldraw (5,1) circle [radius=0.1];
				\draw[below=5pt] (5,1) node {$e_9$};
				\filldraw (7,1) circle [radius=0.1];
				\draw (7,1) circle [radius=0.2];
				\draw[below=5pt] (7,1) node {$e_{10},e_{10}q$};
				\filldraw (9,1) circle [radius=0.1];
				\draw (9,1) circle [radius=0.2];
				\draw[below=5pt] (9,1) node {$e_{11},e_{11}q$};
				\filldraw (3,7) circle [radius=0.1];
				\draw[above=5pt] (3,7) node {$e_{12}$};
				\filldraw (5,7) circle [radius=0.1];
				\draw[above=5pt] (5,7) node {$e_{13}$};
				\filldraw (7,7) circle [radius=0.1];
				\draw[above=5pt] (7,7) node {$e_{14}$};
				\filldraw (9,7) circle [radius=0.1];
				\draw (9,7) circle [radius=0.2];
				\draw (9,7) circle [radius=0.3];
				\draw[above=5pt] (9,7) node {$e_{15},e_{15}q^{-1},e_{15}q^{-2}$};
			\end{tikzpicture}
		\end{center}
		The $q$-Fuchs relation of this example is $e_1e_2e_3^2q\cdot e_{12}e_{13}e_{14}e_{15}^3q^{-3}=e_4e_5e_6e_7\cdot e_8e_9e_{10}^2qe_{11}^2q$.
	\end{exa}
	We consider transforming the $q$-difference equation \eqref{trigeq} to a scalar $q$-difference equation of rank $3$.	
	
	\begin{prop}\label{propsc}
		We suppose that $y={}^\mathrm{T}(y_1,y_2,y_3)$ satisfies the equation \eqref{trigeq}.
		Then the function $z=((-x/e_1)_\infty(-x/e_2)_\infty(-x/e_3)_\infty)^{-1} y_1$ satisfies a Fuchsian $q$-difference equation 
		\begin{align}\label{scalareq}
			Lz=\sum_{i=0}^7 \sum_{j=0}^3 a_{i,j}x^i T_x^jz=\sum_{j=0}^3 P_j(x)T_x^jz=0,
		\end{align}
		which has the following point configuration.
		In addition there is $c$ such that
		\begin{align}\label{apparent}
			P_3(f/q)=c P_2(f),\ P_2(f/q)=c P_1(f),\ P_1(f/q)=c P_0(f).
		\end{align}
		Here $f$ is a suitable value determined by components of $A$.
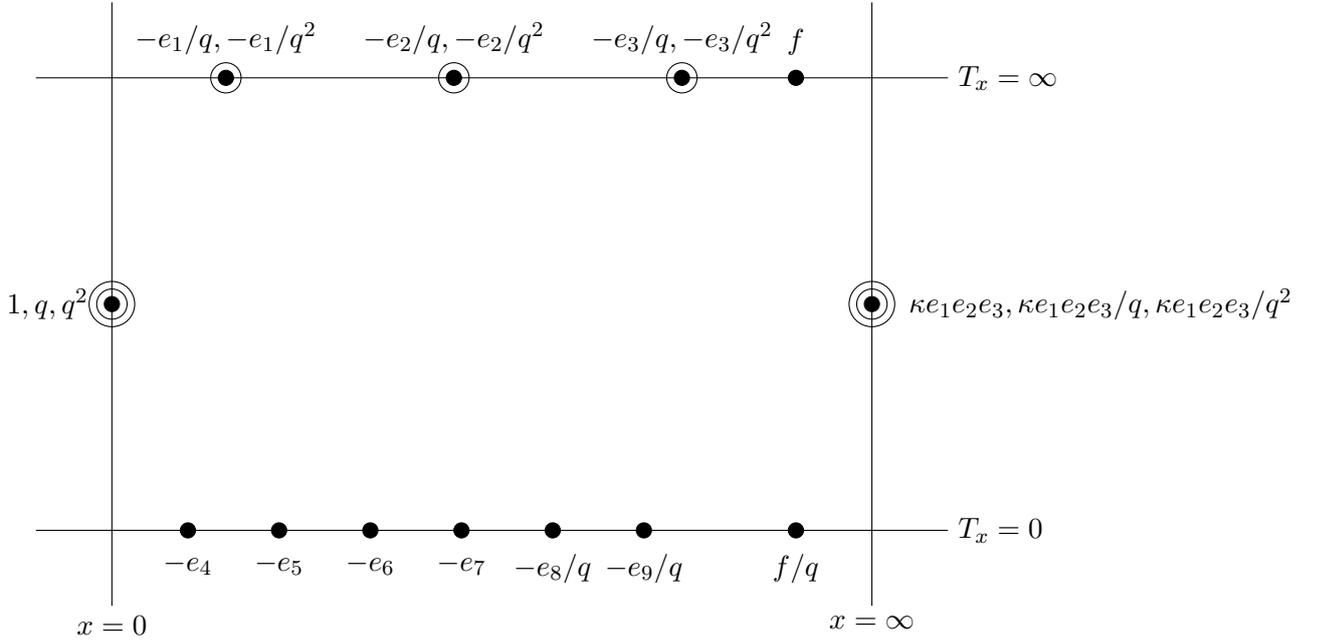
		\begin{center}
			\begin{tikzpicture}
				%\draw[help lines] (0,0) grid (13,8);
				\draw (0,1)--(12,1);
				\draw (0,7)--(12,7);
				\draw (1,0)--(1,8);
				\draw (11,0)--(11,8);
				\draw[right] (12,1) node {$T_x=0$};
				\draw[right] (12,7) node {$T_x=\infty$};
				\draw[below] (1,0) node {$x=0$};
				\draw[below] (11,0) node {$x=\infty$};
				\filldraw (1,4) circle [radius=0.1];
				\draw (1,4) circle [radius=0.3];
				\draw (1,4) circle [radius=0.2];
				\draw[left=5pt] (1,4) node {$1,q,q^2$};
				\filldraw (11,4) circle [radius=0.1];
				\draw (11,4) circle [radius=0.2];
				\draw (11,4) circle [radius=0.3];
				\draw[right=10pt] (11,4) node {$\kappa e_1e_2e_3,\kappa e_1e_2e_3/q,\kappa e_1e_2e_3/q^{2}$};
				\filldraw (2,1) circle [radius=0.1];
				\draw[below=5pt] (2,1) node {$-e_4$};
				\filldraw (3.2,1) circle [radius=0.1];
				\draw[below=5pt] (3.2,1) node {$-e_5$};
				\filldraw (4.4,1) circle [radius=0.1];
				\draw[below=5pt] (4.4,1) node {$-e_6$};
				\filldraw (5.6,1) circle [radius=0.1];
				\draw[below=5pt] (5.6,1) node {$-e_7$};
				\filldraw (6.8,1) circle [radius=0.1];
				\draw[below=5pt] (6.8,1) node {$-e_8/q$};
				\filldraw (8,1) circle [radius=0.1];
				\draw[below=5pt] (8,1) node {$-e_9/q$};
				\filldraw (10,1) circle [radius=0.1];
				\draw[below=5pt] (10,1) node {$f/q$};
				\filldraw (2.5,7) circle [radius=0.1];
				\draw (2.5,7) circle [radius=0.2];
				\draw[above=5pt] (2.5,7) node {$-e_1/q,-e_1/q^2$};
				\filldraw (5.5,7) circle [radius=0.1];
				\draw (5.5,7) circle [radius=0.2];
				\draw[above=5pt] (5.5,7) node {$-e_2/q,-e_2/q^2$};
				\filldraw (8.5,7) circle [radius=0.1];
				\draw (8.5,7) circle [radius=0.2];
				\draw[above=5pt] (8.5,7) node {$-e_3/q,-e_3/q^2$};
				\filldraw (10,7) circle [radius=0.1];
				\draw[above=5pt] (10,7) node {$f$};
			\end{tikzpicture}
		\end{center}
		
	\end{prop}
	\begin{proof}
		See Appendix \ref{appC}.
	\end{proof}
	\begin{rem}
		Imitating differential cases, we call roots of the highest/lowest coefficient on $T_x$ for a scalar linear $q$-difference operator ``singularities''.
		When we construct a scalar $q$-difference equation from a system $T_xy=Ay$ of first order $q$-difference equations, roots of $\det A$ (or $\det T_x^{\pm 1}A$, $\det T_x^{\pm 2}A,\ldots$) appears as singularities. 
		Since $f q^{\mathbb{Z}}$ does not have a root of $\det A$, $f$ is an ``apparent singularity'' of the $q$-difference equation \eqref{scalareq}.
		The condition \eqref{apparent} would give a condition of apparent singularity.
		We remark that the notion of apparent singularities is introduced and several properties are discussed in \cite{FJM}.
		See also \cite{Park}.
	\end{rem}
	In \cite{MY}, a quantum curve with $W(E_8^{(1)})$-symmetry is discussed and quantum representation of $W(E_8^{(1)})$ is given by transformations of this quantum curve.
	Let $x$,$y$ be $q$-commutative variables, i.e. $yx=qxy$.
	The quantum curve $C$ is given as follows:
	\begin{itemize}
		\item $C$ : $F(x,y)=0$.
		\item $F(x,y)$ : \mbox{polynomial of $x$ and $y$, $\deg_x F(x,y)=6$, $\deg_y F(x,y)=3$}.
		\item The curve $C$ passes through the following $11$ points with corresponding multiplicities:
		\begin{align}
			&(-e_i,0)\mbox{ : multiplicity $1$}\ (i=1,2,3,4,5,6),\\
			&\left(-\frac{h_1}{e_i},\infty\right)\mbox{ : multiplicity $2$}\ (i=7,8,9),\\
			&\left(\infty,-\frac{e_{10}}{h_2}\right)\mbox{ : multiplicity $3$},\\
			&\left(0,-\frac{1}{e_{11}}\right)\mbox{ : multiplicity $3$}
		\end{align}
	\end{itemize}
	Since the variables $x$ and $y$ is $q$-commutative, the polynomial $F(x,y)$ can be regarded as a linear $q$-difference operator with respect to $x$.
	In this sense, the curve $C$ can be regarded as a $q$-difference equation.
	The $q$-difference equation in Proposition \ref{propsc} is a non-autonomous cases of the curve $C$ with suitable changes of parameters and some gauge transformation.
	\begin{rem}\label{remsc}
		In the context of Painlev\'e equations, the quantum curve often appears as a specialization of the Lax linear equation for non-autonomous Painlev\'e equation.
		Although we do not construct a deformation equation of \eqref{defeq} here, it would be expected that the $q$-Painlev\'e equation of type $E_8^{(1)}$ is derived by a compatibility condition of \eqref{defeq} and its deformation.
		We note that if we specialize one of two accessory parameters as $f=-e_7$, operator $L$ \eqref{scalareq} reduces to a polynomial bi-degree $(6,3)$.
		This specialization recovers the original Moriyama--Yamada curve.
		In other words, the specialization $f=-e_7$ is an autonomization of our equation.
	\end{rem}

	\section{Summary and discussions}\label{secsumm}
	In this paper, we constructed a $q$-difference equation \eqref{defeq} which has the affine Weyl group symmetry of type $E_8^{(1)}$.
	We also found that one of the generators is given by using the $q$-middle convolution, and the others are adjacent transpositions of parameters.
	In Section \ref{secqmc}, we introduced the $q$-middle convolution via $q$-Okubo type equation.
	This formulation clarifies a condition of path.
	The main result of this paper is Theorem \ref{E8mc} which gives a transformation of the equation \eqref{defeq} by means of the $q$-middle convolution.
	It follows from this result that the equation has $W(E_8^{(1)})$-symmetry (see Proposition \ref{propdynkin}).
	In Section \ref{secsc}, we reduced our equation \eqref{defeq} to the scalar equation \eqref{scalareq} and characterized it by the point configuration.
	This scalar equation is a non-autonomous extension of Moriyama--Yamada's curve \cite{MY} which appears in the theory of topological strings and is applied to the quantum Painlev\'e equation.
	 Therefore this opens the possibility of applying our results to quantum integrable systems.
	
	It is important to give $s_3$-action discussed in Conjecture \ref{conjrepE8}.
	It is also important to describe these conditions by means of some assumptions of the matrix $A$ \eqref{defeq}.
	There are also many related problems.
	\begin{itemize}
		\item[1.] The equation \eqref{defeq} has $W(E_8^{(1)})$-symmetry.
		It is fundamental and very important to construct an explicit expression of \eqref{defeq} and its deformation equation.
		If these are constructed, we might get a Lax formalization of the $q$-Painlev\'e equation of type $E_8^{(1)}$ (see also Remark \ref{remsc}).
		It is also interesting to compare established Lax pairs of $q$-$P(E_8^{(1)})$ (cf. \cite{Yamada}).
		We remark that the Lax linear equation in \cite{Yamada} is second order, although our equation is third order.
		See also \cite{NRY,Takemuradeg} for related topics with the Lax equation in \cite{Yamada}.
		\item[2.] To give special solutions of \eqref{defeq} is interesting.
		Due to the characterization \eqref{spectral}, the equation \eqref{defeq} can be regarded as a $q$-analog of rank $3$ Fuchsian differential equation with $3$ singularities on $\mathbb{P}^{1}=\mathbb{C}\cup\{\infty\}$.
		This differential equation admits special solutions in terms of the generalized hypergeometric series ${}_3F_2$ and the Dotsenko--Fateev integral \cite{DF}.
		In this sense it is expected that the equation \eqref{defeq} has special solutions by some $q$-hypergeometric series and integrals.
		For the $q$-Selberg type integral including a $q$-analog of Dotsenko--Fateev integrals, and its properties, see \cite{Ito2020,Ito2023}.
		A special solution of \eqref{defeq} in terms of the very-well poised $q$-hypergeometric series ${}_{10}W_9$ is obtained.
		This result is discussed in the next paper.
		\item[3.] In Sakai's classification \cite{Sakai}, the $q$-Painlev\'e equation $q$-$P(E_8^{(1)})$ degenerates to it of type $E_7^{(1)}$, $E_6^{(1)}$, $D_5^{(1)}$ and several equations.
		It is very important to discuss degenerations of \eqref{defeq}.
		In \cite{Park}, a Lax pair of $q$-$P(E_6^{(1)})$ is given by a  rank $3$ $q$-difference equation of the form
		\begin{align}
			T_xy=(A^{(0)}+xA^{(1)})y,
		\end{align}
		and its deformation equation.
		This rank $3$ equation is characterized by the spectral type $(111;111;111)$, thus this is related to a degeneration of our equation \eqref{defeq}.
	\end{itemize}
	%特殊解, 変形方程式, 具体系, 退化
	
	\appendix
	\section{Proof of Theorem \ref{E8mc}}\label{appA}
	In this appendix, we prove Theorem \ref{E8mc}.
	\begin{proof}[Proof of Theorem \ref{E8mc}]
		The proof proceeds by analyzing the invariant subspaces  $\mathcal{K}$ and $\mathcal{L}$, followed by an explicit computation of the transformed coefficient matrix.
		First, by the gauge transformation on $g$, the equation \eqref{defeq} is transformed to
		\begin{align}\label{gauge}
			T_xy=A'y,\quad A'=\frac{A}{(1+x/e_1)(1+x/e_2)(1+x/e_3)}.
		\end{align}
		Rewriting this equation by using $D_x$, since $A=I+O(x)$ at $x=0$, we get
		\begin{align}\label{DxAi}
			D_xy=\sum_{i=1}^3 \frac{A_i}{x+e_i}y,\quad \sum_{i=1}^3 \frac{A_i}{x+e_i}=\frac{I-A'}{(1-q)x},
		\end{align}
		where $A_i$ is a constant matrix.
		We put
		\begin{align}
			\label{defG1}G_1=\begin{bmatrix}
				q^\lambda A_1+[\lambda]I&q^\lambda A_2&q^\lambda A_3\\O&O&O\\O&O&O
			\end{bmatrix},
			\\
			\label{defG2}G_2=\begin{bmatrix}
				O&O&O\\
				q^\lambda A_1&q^\lambda A_2+[\lambda]I&q^\lambda A_3\\O&O&O
			\end{bmatrix},
			\\
			\label{defG3}G_3=\begin{bmatrix}
				O&O&O\\
				O&O&O\\
				q^\lambda A_1&q^\lambda A_2&q^\lambda A_3+[\lambda]I
			\end{bmatrix}.
		\end{align}
		We consider the invariant spaces
		\begin{align}
			\mathcal{K}=\begin{bmatrix}
				\mathrm{Ker}(A_1)\\\mathrm{Ker}(A_2)\\\mathrm{Ker}(A_3)
			\end{bmatrix},\quad \mathcal{L}=\mathrm{Ker}(G_1+G_2+G_3),
		\end{align}
		in particular we discuss the dimensions of $\mathcal{K}$ and $\mathcal{L}$.
		Due to the construction of $A_i$ \eqref{DxAi}, we have
		\begin{align}
			\notag A_i&=\lim_{x\to-e_i}\frac{I-A'}{(1-q)x}(x+e_i)
			\\
			\notag&=\lim_{x\to-e_i}\frac{I-A(x)/((1+x/e_1)(1+x/e_2)(1+x/e_3))}{(1-q)x}(x+e_i)
			\\
			&\propto A(-e_i).
		\end{align}
		Since $\displaystyle \det A(x)=\kappa^3 \prod_{i=1}^9 (x+e_i)$ and $e_i\neq e_j$ $(i\neq j)$, we find $\dim\mathrm{Ker}(A_i)=1$.
		We put $a_i\in \mathrm{Ker}(A_i)$ as $a_i\neq 0$, and
		\begin{align}
			k_1=\begin{bmatrix}
				a_1\\0\\0
			\end{bmatrix},
			{k}_2=\begin{bmatrix}
				0\\a_2\\0
			\end{bmatrix},
			k_3=\begin{bmatrix}
				0\\0\\a_3
			\end{bmatrix},
		\end{align}
		then $\{k_1, k_2,k_3\}$ form a basis of $\mathcal{K}$.
		In particular we have $\dim\mathcal{K}=3$.
		We calculate the eigenpolynomial of $G_1+G_2+G_3$ as follows:
		\begin{align}
			\notag&\det (G_1+G_2+G_3-z I)
			\\
			\notag&=\det \left(q^\lambda \begin{bmatrix}
				A_1&A_2&A_3\\A_1&A_2&A_3\\A_1&A_2&A_3
			\end{bmatrix}+([\lambda]-z)\begin{bmatrix}
				I&O&O\\O&I&O\\O&O&I
			\end{bmatrix}\right)
			\\
			\notag&=\det \left(q^\lambda \begin{bmatrix}
				A_1&A_2&A_3\\O&O&O\\O&O&O
			\end{bmatrix}+([\lambda]-z)\begin{bmatrix}
				I&O&O\\-I&I&O\\-I&O&I
			\end{bmatrix}\right)
			\\
			\notag&=\det \left(q^\lambda \begin{bmatrix}
				A_1+A_2+A_3&A_2&A_3\\O&O&O\\O&O&O
			\end{bmatrix}+([\lambda]-z)\begin{bmatrix}
				I&O&O\\O&I&O\\O&O&I
			\end{bmatrix}\right)
			\\
			&=([\lambda]-z)^6 \det (q^\lambda (A_1+A_2+A_3)+([\lambda]-z)I).
		\end{align}
		Due to \eqref{DxAi}, we have
		\begin{align}
			\notag A_1+A_2+A_3&=\lim_{x\to\infty}\frac{I-A'}{1-q}
			\\
			&=\lim_{x\to\infty}\frac{I-A(x)/((1+x/e_1)(1+x/e_2)(1+x/e_3))}{1-q}=\frac{1-\kappa e_1e_2e_3}{1-q}I.
		\end{align}
		Therefore we get
		\begin{align}
			\det (q^\lambda (A_1+A_2+A_3)+([\lambda]-z)I)=\left(\frac{1-q^\lambda \kappa e_1e_2 e_3}{1-q}-z\right)^3.
		\end{align}
		We put $q^\lambda =(\kappa e_1e_2e_3)^{-1}$, then we have $\dim \mathcal{L}\leq 3$.
		On the other hand we have
		\begin{align}
			\notag&(G_1+G_2+G_3)\begin{bmatrix}
				v\\v\\v
			\end{bmatrix}\quad (v\in\mathbb{C}^3)
			\\
			\notag&=\begin{bmatrix}
				(q^\lambda (A_1+A_2+A_3)+[\lambda]I)v\\
				(q^\lambda (A_1+A_2+A_3)+[\lambda]I)v\\
				(q^\lambda (A_1+A_2+A_3)+[\lambda]I)v
			\end{bmatrix}
			\\
			&=\begin{bmatrix}
				Ov\\Ov\\Ov
			\end{bmatrix}=0.
		\end{align}
		Therefore we have $\mathcal{L}=\{{}^T[v,v,v]\mid v\in\mathbb{C}^3\}$, $\dim\mathcal{L}=3$ and $\dim(\mathcal{K}+\mathcal{L})=6$.
		Let $l_1,l_2,l_3$ be a basis of $\mathcal{L}$ and $u_1,u_2,u_3$ be a basis of the complementary space of $\mathcal{K}+\mathcal{L}$ in $\mathbb{C}^9$.
		We have
		\begin{align}
			R^{-1}G_i R=\begin{bmatrix}
				\tilde{G}_i&O\\
				\ast&H_i
			\end{bmatrix},\quad R=[u_1,u_2,u_3,k_1,k_2,k_3,l_1,l_2,l_3],
		\end{align}
		where $\tilde{G}_i$ is a $3\times 3$ matrix.
		Using $G_i$, we get a equation
		\begin{align}
			D_xy=\sum_{i=1}^3\frac{\tilde{G}_i}{x+e_i}y,
		\end{align}
		by the $q$-middle convolution $mc_\lambda$.
		Rewriting this equation we have
		\begin{align}
			T_xy=\left[I-(1-q)x\sum_{i=1}^3\frac{\tilde{G}_i}{x+e_i}\right]y.
		\end{align}
		Applying the gauge transformation by $g^{-1}$, we finally get
		\begin{align}
			T_xy=\tilde{A}y,\quad \tilde{A}=\tilde{A}(x)=(1+x/e_1)(1+x/e_2)(1+x/e_3)\left[I-(1-q)x\sum_{i=1}^3\frac{\tilde{G}_i}{x+e_i}\right].
		\end{align}
		It is obvious that $\tilde{A}$ is a polynomial of degree three in $x$.
		
		In the following we show properties of $\tilde{A}$, i.e. we prove that $\tilde{A}$ is in the form of \eqref{Apro1} and satisfy \eqref{Apro2}.
		Clearly we have $\tilde{A}(0)=I$.
		The highest term of $\tilde{A}$ is
		\begin{align}
			\lim_{x\to\infty} x^{-3}\tilde{A}(x)=\frac{1}{e_1e_2e_3}\left[I-(1-q)\sum_{i=1}^3\tilde{G}_i\right].
		\end{align}
		We calculate $\tilde{G}_1+\tilde{G}_2+\tilde{G}_3$ in the following.
		We remark that
		\begin{align}
			R^{-1}(G_1+G_2+G_3)R=\begin{bmatrix}
				\tilde{G}_1+\tilde{G}_2+\tilde{G}_3&O\\
				\ast&H_1+H_2+H_3
			\end{bmatrix}.
		\end{align}
		Since 
		\begin{align}
			R=[{u}_1,{u}_2,{u}_3,{k}_1,{k}_2,{k}_3,{l}_1,{l}_2,{l}_3],\quad {l}_i\in\mathcal{L}=\mathrm{Ker}(G_1+G_2+G_3),
		\end{align}
		we have
		\begin{align}
			R^{-1}(G_1+G_2+G_3)R=[\ast,\ast,\ast,\ast,\ast,\ast,{0},{0},{0}].
		\end{align}
		Summarizing these, we obtain
		\begin{align}
			R^{-1}(G_1+G_2+G_3)R=\begin{bmatrix}
				\tilde{G}_1+\tilde{G}_2+\tilde{G}_3&O_{3\times 3}&O_{3\times3}\\
				J_1&J_2&O_{3\times 3}\\
				J_3&J_4&O_{3\times3}
			\end{bmatrix}.
		\end{align}
		We recall that $G_1+G_2+G_3$ has the eigenvalue $[\lambda]$ with multiplicity $6$.
		We also have
		\begin{align}
			G_1+G_2+G_3-[\lambda]I=q^\lambda\begin{bmatrix}
				A_1&A_2&A_3\\A_1&A_2&A_3\\A_1&A_2&A_3
			\end{bmatrix}\to q^\lambda \begin{bmatrix}
				A_1&A_2&A_3\\O&O&O\\O&O&O
			\end{bmatrix},
		\end{align}
		by the elementary row operations.
		Therefore the dimension of eigenspace of $[\lambda]$ is $6$
		This leads us to
		\begin{align}\label{tildeGsum}
			\begin{bmatrix}
				\tilde{G}_1+\tilde{G}_2+\tilde{G}_3&O_{3\times 3}\\
				J_1&J_2
			\end{bmatrix}=[\lambda]I,
		\end{align}
		in particular we have$\tilde{G}_1+\tilde{G}_2+\tilde{G}_3=[\lambda]I$.
		In conclusion, the highest term of $\tilde{A}$ is
		\begin{align}
			\frac{1}{e_1e_2e_3}\left[I-(1-q)\sum_{i=1}^3\tilde{G}_i\right]=\frac{1}{e_1e_2e_3}(1-(1-q)[\lambda])I=\frac{q^\lambda}{e_1e_2e_3}I=\frac{1}{\kappa(e_1e_2e_3)^2}I=\tilde{\kappa}I.
		\end{align}
		This completes the proof that $\tilde{A}$ is in the form of \eqref{Apro1}.

		We prove \eqref{Apro2} in the following.
		We first show 
		\begin{align}\label{H}
			H_1=\begin{bmatrix}
				[\lambda]&0&0&0&0&0\\
				0&0&0&0&0&0\\
				0&0&0&0&0&0\\
				0&0&0&0&0&0\\
				0&0&0&0&0&0\\
				0&0&0&0&0&0
			\end{bmatrix},\quad 
			H_2=\begin{bmatrix}
				0&0&0&0&0&0\\
				0&[\lambda]&0&0&0&0\\
				0&0&0&0&0&0\\
				0&0&0&0&0&0\\
				0&0&0&0&0&0\\
				0&0&0&0&0&0
			\end{bmatrix},\quad
			H_3=\begin{bmatrix}
				0&0&0&0&0&0\\
				0&0&0&0&0&0\\
				0&0&[\lambda]&0&0&0\\
				0&0&0&0&0&0\\
				0&0&0&0&0&0\\
				0&0&0&0&0&0
			\end{bmatrix}.
		\end{align}
		We only show the case of $H_1$ here.
		Other cases can be proved in a similar manner.
		We divide a $9\times 6$ matrix $\begin{bmatrix}
			O_{3\times 6}\\H_1
		\end{bmatrix}$ into column vectors as
		\begin{align}
			\begin{bmatrix}
				O\\H_1
			\end{bmatrix}=[{h}_1,{h}_2,{h}_3,{h}_4,{h}_5,{h}_6].
		\end{align}
		We recall that
		\begin{align}
			R^{-1}G_1R=\begin{bmatrix}
				\tilde{G}_1&O\\\ast&H_1
			\end{bmatrix}=[\ast,\ast,\ast,{h}_1,{h}_2,{h}_3,{h}_4,{h}_5,{h}_6].
		\end{align}
		Since
		\begin{align}
			&G_1=\begin{bmatrix}
				q^\lambda A_1+[\lambda]I&q^\lambda A_2&q^\lambda A_3\\
				O&O&O\\O&O&O
			\end{bmatrix},
			\\
			&R=[{u}_1,{u}_2,{u}_3,{k}_1,{k}_2,{k}_3,{l}_1,{l}_2,{l}_3],
			\\
			&{k}_1=\begin{bmatrix}
				{a}_1\\{0}\\{0}
			\end{bmatrix},
			{k}_2=\begin{bmatrix}
				{0}\\{a}_2\\{0}
			\end{bmatrix},
			{k}_3=\begin{bmatrix}
				{0}\\{0}\\{a}_3
			\end{bmatrix},\quad {a}_i\in\mathrm{Ker}(A_i),
		\end{align}
		we have
		\begin{align}
			G_1R=[\ast,\ast,\ast,[\lambda]{k}_1,{0},{0},\ast,\ast,\ast].
		\end{align}
		Due to the definition of $R$, we find
		\begin{align}
			R^{-1}G_1R=[\ast,\ast,\ast,[\lambda]{p}_4,{0},{0},\ast,\ast,\ast],
		\end{align}
		where ${p}_i$ is the standard unit vector in $\mathbb{C}^9$.
		Hence, we get 
		\begin{align}\label{h1}
			[{h}_1,{h}_2,{h}_3]=[[\lambda]{p}_4,{0},{0}].
		\end{align}
		Since
		\begin{align}
			G_1=\begin{bmatrix}
				\ast&\ast&\ast\\O&O&O\\O&O&O
			\end{bmatrix},\quad G_2=\begin{bmatrix}
				O&O&O\\\ast&\ast&\ast\\O&O&O
			\end{bmatrix},\quad G_3=\begin{bmatrix}
				O&O&O\\O&O&O\\\ast&\ast&\ast
			\end{bmatrix},
		\end{align}
		we have
		\begin{align}
			G_1R=\begin{bmatrix}
				\ast&\ast&F_1\\O&O&O\\O&O&O
			\end{bmatrix},\quad G_2R=\begin{bmatrix}
				O&O&O\\\ast&\ast&F_2\\O&O&O
			\end{bmatrix},\quad G_3R=\begin{bmatrix}
				O&O&O\\O&O&O\\\ast&\ast&F_3
			\end{bmatrix},
		\end{align}
		where $F_i$ is a suitable $3\times3$ matrix.
		Thus we have
		\begin{align}
			(G_1+G_2+G_3)R=\begin{bmatrix}
				\ast&\ast&F_1\\\ast&\ast&F_2\\\ast&\ast&F_3
			\end{bmatrix}
		\end{align}
		By the definition of $R$:
		\begin{align}
			R=[{u}_1,{u}_2,{u}_3,{k}_1,{k}_2,{k}_3,{l}_1,{l}_2,{l}_3],\quad {l}_i\in\mathcal{L}=\mathrm{Ker}(G_1+G_2+G_3),
		\end{align}
		we have $F_1=F_2=F_3=O$.
		Hence, we obtain 
		\begin{align}
			\begin{bmatrix}
				\tilde{G}_1&O\\\ast&H_1
			\end{bmatrix}=R^{-1}G_1R=R^{-1}\begin{bmatrix}
				\ast&\ast&F_1\\O&O&O\\O&O&O
			\end{bmatrix}=\begin{bmatrix}
				\ast&\ast&O\\\ast&\ast&O\\\ast&\ast&O
			\end{bmatrix},
		\end{align}
		in particular we have
		\begin{align}\label{h2}
			[{h}_4,{h}_5,{h}_6]=[{0},{0},{0}].
		\end{align}
		Due to \eqref{h1}, \eqref{h2}, we finally
		\begin{align}
			H_1=\begin{bmatrix}
				[\lambda]&0&0&0&0&0\\
				0&0&0&0&0&0\\
				0&0&0&0&0&0\\
				0&0&0&0&0&0\\
				0&0&0&0&0&0\\
				0&0&0&0&0&0
			\end{bmatrix}.
		\end{align}
		This completes the proof of \eqref{H}.
		
		We recall that
		\begin{align}
			R^{-1}G_i R=\begin{bmatrix}
				\tilde{G}_i&O\\
				\ast&H_i
			\end{bmatrix},\quad \tilde{A}=(1+x/e_1)(1+x/e_2)(1+x/e_3)\left[I-(1-q)x\sum_{i=1}^3\frac{\tilde{G}_i}{x+e_i}\right].
		\end{align}
		We calculate $\displaystyle \det \left[I-(1-q)x\sum_{i=1}^3\frac{{G}_i}{x+e_i}\right]$ as follows:
		\begin{align}
			\notag&\det \left[I-(1-q)x\sum_{i=1}^3\frac{{G}_i}{x+e_i}\right]
			\\
			\notag&=\det \left[I-(1-q)x\begin{bmatrix}
				\dfrac{q^\lambda A_1+[\lambda]}{x+e_1}&\dfrac{q^\lambda A_2}{x+e_1}&\dfrac{q^\lambda A_3}{x+e_1}
				\\
				\dfrac{q^\lambda A_1}{x+e_2}&\dfrac{q^\lambda A_2+[\lambda]}{x+e_2}&\dfrac{q^\lambda A_3}{x+e_2}
				\\
				\dfrac{q^\lambda A_1}{x+e_3}&\dfrac{q^\lambda A_2}{x+e_3}&\dfrac{q^\lambda A_3+[\lambda]}{x+e_3}
			\end{bmatrix}\right]
			\\
			\notag&=\frac{1}{((x+e_1)(x+e_2)(x+e_3))^3}
			\\
			\notag&\quad\times\det \left[\begin{bmatrix}
				x+e_1&&\\&x+e_2&\\&&x+e_3
			\end{bmatrix}-(1-q)[\lambda]I-(1-q)xq^\lambda\begin{bmatrix}
				A_1&A_2&A_3\\A_1&A_2&A_3\\A_1&A_2&A_3
			\end{bmatrix}\right]
			\\
			\notag&=\frac{1}{((x+e_1)(x+e_2)(x+e_3))^3}
			\\
			\notag&\quad\times\det\left[\begin{bmatrix}
				q^\lambda x+e_1&&\\&q^\lambda x+e_2&\\&&q^\lambda x+e_3
			\end{bmatrix}-(1-q)xq^\lambda\begin{bmatrix}
				A_1&A_2&A_3\\A_1&A_2&A_3\\A_1&A_2&A_3
			\end{bmatrix}\right]
			\\
			\notag&=\frac{1}{((x+e_1)(x+e_2)(x+e_3))^3}
			\\
			\notag&\quad\times\det \left[\begin{bmatrix}
				q^\lambda x+e_1&&\\-(q^\lambda x+e_1)&q^\lambda x+e_2&\\-(q^\lambda x+e_1)&&q^\lambda x+e_3
			\end{bmatrix}-(1-q)xq^\lambda\begin{bmatrix}
				A_1&A_2&A_3\\&&\\&&
			\end{bmatrix}\right]
			\\
			\notag&=\frac{((q^\lambda x+e_1)(q^\lambda x+e_2)(q^\lambda x+e_3))^3}{((x+e_1)(x+e_2)(x+e_3))^3}
			\\
			\notag&\quad\times\det \left[\begin{bmatrix}
				I&&\\-I&I&\\-I&&I
			\end{bmatrix}-(1-q)xq^\lambda\begin{bmatrix}
				\dfrac{A_1}{q^\lambda x+e_1}&\dfrac{A_2}{q^\lambda x+e_2}&\dfrac{A_3}{q^\lambda x+e_3}\\&&\\&&
			\end{bmatrix}\right]
			\\
			\notag&=\frac{((q^\lambda x+e_1)(q^\lambda x+e_2)(q^\lambda x+e_3))^3}{((x+e_1)(x+e_2)(x+e_3))^3}
			\\
			\notag&\quad\times\det\left[\begin{bmatrix}
				I&&\\&I&\\&&I
			\end{bmatrix}-(1-q)xq^\lambda \begin{bmatrix}
				\displaystyle \sum_{i=1}^3\dfrac{A_i}{q^\lambda x+e_i}&\dfrac{A_2}{q^\lambda x+e_2}&\dfrac{A_3}{q^\lambda x+e_3}\\&&\\&&
			\end{bmatrix}\right]
			\\
			&=\frac{((q^\lambda x+e_1)(q^\lambda x+e_2)(q^\lambda x+e_3))^3}{((x+e_1)(x+e_2)(x+e_3))^3}\det\left[I-(1-q)xq^\lambda \displaystyle \sum_{i=1}^3\dfrac{A_i}{q^\lambda x+e_i}\right].
		\end{align}
		Since
		\begin{align}
			&\sum_{i=1}^3 \frac{A_i}{x+e_i}=\frac{I-A'(x)}{(1-q)x},\quad A'(x)=\dfrac{A(x)}{(1+x/e_1)(1+x/e_2)(1+x/e_3)},\\
			&\det A(x)=\kappa^3\prod_{i=1}^9(x+e_i),
		\end{align}
		we obtain
		\begin{align}
			&\notag\det\left[I-(1-q)xq^\lambda \displaystyle \sum_{i=1}^3\dfrac{A_i}{q^\lambda x+e_i}\right]=\det A'(q^\lambda x)\\
			&=\frac{1}{((1+q^\lambda x/e_1)(1+q^\lambda x/e_2)(1+q^\lambda x/e_3))^3}\kappa^3\prod_{i=1}^9(q^\lambda x+e_i).
		\end{align}
		It follows that
		\begin{align}\label{detG}
			\det \left[I-(1-q)x\sum_{i=1}^3\frac{{G}_i}{x+e_i}\right]=\frac{(\kappa e_1e_2e_3)^3}{((x+e_1)(x+e_2)(x+e_3))^3}\prod_{i=1}^9(q^\lambda x+e_i)
		\end{align}
		On the other hand,  we have
		\begin{align}
			&\det \left[I-(1-q)x\sum_{i=1}^3\frac{{G}_i}{x+e_i}\right]=\det R^{-1}\left[I-(1-q)x\sum_{i=1}^3\frac{{G}_i}{x+e_i}\right]R
			\notag
			\\
			&=\det\left[I-(1-q)x\sum_{i=1}^3\dfrac{\begin{bmatrix}
					\tilde{G}_i&O\\
					\ast&H_i
			\end{bmatrix}}{x+e_i}\right]
			\notag
			\\
			&=\det\left[I-(1-q)x\sum_{i=1}^3\dfrac{\tilde{G}_i}{x+e_i}\right]\cdot \det\left[I-(1-q)x\sum_{i=1}^3\dfrac{H_i}{x+e_i}\right].
		\end{align}
		Using the expression \eqref{H} of $H_i$, we get 
		\begin{align}\label{detH}
			\det\left[I-(1-q)x\sum_{i=1}^3\dfrac{H_i}{x+e_i}\right]=\prod_{i=1}^3 \left(1-(1-q)x\frac{[\lambda]}{x+e_i}\right)=\prod_{i=1}^3\frac{q^\lambda x+e_i}{x+e_i}.
		\end{align}
		Summarizing \eqref{detG} and \eqref{detH}, we have
		\begin{align}\label{dettildeAapp}
			\det\left[I-(1-q)x\sum_{i=1}^3\dfrac{\tilde{G}_i}{x+e_i}\right]=\frac{(\kappa e_1e_2e_3)^3}{((x+e_1)(x+e_2)(x+e_3))^{2}}\prod_{i=4}^9(q^\lambda x+e_i).
		\end{align}
		Since
		\begin{align}
			\tilde{A}=(1+x/e_1)(1+x/e_2)(1+x/e_3)\left[I-(1-q)x\sum_{i=1}^3\dfrac{\tilde{G}_i}{x+e_i}\right],
		\end{align}
		we finally find
		\begin{align}
			\notag	\det\tilde{A}&=((1+x/e_1)(1+x/e_2)(1+x/e_3))^3\cdot \frac{(\kappa e_1e_2e_3)^3}{((x+e_1)(x+e_2)(x+e_3))^{2}}\prod_{i={4}}^9(q^\lambda x+e_i)\\&=\tilde\kappa^3\prod_{i=1}^9(x+\tilde{e}_i).
		\end{align}
		This completes the proof of Theorem \ref{E8mc}.
	\end{proof}
	\section{Proof of Proposition \ref{propbraid}}\label{appB}
	In this appendix, we prove Proposition \ref{propbraid}.
	\begin{proof}[proof of Proposition \ref{propbraid}]
		We first prove \eqref{braid0i}.
		Since $s_i.g=g$ ($i\neq 3$) and $s_i.\lambda=\lambda$ ($q^\lambda=(\kappa e_1e_2e_3)^{-1}$), we have $s_0.s_i.A=g\circ mc_\lambda\circ g.(A|_{e_{i}\leftrightarrow e_{i+1}})$ and $s_i.s_0.A=g. (mc_\lambda \circ g.A)|_{e_i\leftrightarrow e_{i+1}}$.
		We put $A_j$ as \eqref{DxAi}.
		We define $G_j$ as \eqref{defG1}, \eqref{defG2} and \eqref{defG3}.
		It follows from calculations in Appendix \ref{appA}, there is a matrix $R$ such that
		\begin{align}\label{conjuGj}
			R^{-1}G_jR=\begin{bmatrix}
				\tilde{G}_j&O\\
				\ast&H_j
			\end{bmatrix},
		\end{align}
		where $\tilde{G}_j$ is a $3\times 3$ matrix and $H_j$ is given by \eqref{H}.
		Using this $\tilde{G}_j$, $s_i.s_0.A$ is given by
		\begin{align}
			(1+x/e_1)(1+x/e_2)(1+x/e_3)\left(I-(1-q)x\sum_{j=1}^3 \frac{\tilde{G}_j}{x+e_j}\right)\bigg|_{e_{i}\leftrightarrow e_{i+1}}.
		\end{align}
		Next we calculate $s_0.s_i.A$.
		We put $\overline{A} =s_i.A$ and define $\overline{A}_j$, $\overline{G}_j$ as
		\begin{align}\label{DxAivar}
		&\sum_{j=1}^3 \frac{\overline{A}_j}{x+s_i.e_j}=\frac{I-\overline{A}}{(1-q)x},
		\\
		&\label{defG1var}\overline{G}_1=\begin{bmatrix}
			q^\lambda \overline{A}_1+[\lambda]I&q^\lambda \overline{A}_2&q^\lambda \overline{A}_3\\O&O&O\\O&O&O
		\end{bmatrix},
		\\
		&\label{defG2var}\overline{G}_2=\begin{bmatrix}
			O&O&O\\
			q^\lambda \overline{A}_1&q^\lambda \overline{A}_2+[\lambda]I&q^\lambda \overline{A}_3\\O&O&O
		\end{bmatrix},
		\\
		&\label{defG3var}\overline{G}_3=\begin{bmatrix}
			O&O&O\\
			O&O&O\\
			q^\lambda \overline{A}_1&q^\lambda \overline{A}_2&q^\lambda \overline{A}_3+[\lambda]I
		\end{bmatrix}.
		\end{align}
		Due to calculations in Appendix \ref{appA}, there is a matrix $S$ such that
		\begin{align}
			S^{-1}\overline{G}_jS=\begin{bmatrix}
				G'_j&O\\\ast&H_j
			\end{bmatrix}.
		\end{align}
		By definition of the matrices $G_j$ and $\overline{G}_j$, we have
		\begin{align}
			\overline{G}_j=\begin{cases}
				G_j|_{e_i\leftrightarrow e_{i+1}}&(i\neq 1,2)\\
				R_{1,2}^{-1}(G_j|_{e_i\leftrightarrow e_{i+1}})R_{1,2}&(i=1)\\
				R_{2,3}^{-1}(G_j|_{e_i\leftrightarrow e_{i+1}})R_{2,3}&(i=2)
			\end{cases},
		\end{align}
		where
		\begin{align}
			R_{1,2}=\begin{bmatrix}
				O&I&O\\I&O&O\\O&O&I
			\end{bmatrix},\ 
			R_{2,3}=\begin{bmatrix}
				I&O&O\\O&O&I\\O&I&O
			\end{bmatrix}.
		\end{align}
		In particular we have
		\begin{align}
			I-(1-q)x\sum_{j=1}^3\frac{\overline{G}_j}{x+s_i.e_j}\simeq \left[I-(1-q)x\sum_{j=1}^3 \frac{G_j}{x+e_j}\right]\bigg|_{e_{i}\leftrightarrow e_{i+1}}.
		\end{align}
		Using the relation \eqref{conjuGj}, we find
		\begin{align}
			I-(1-q)x\sum_{j=1}^3\frac{\begin{bmatrix}
					{G_j'}&O\\
					\ast&H_j
			\end{bmatrix}}{x+s_i.e_j}\simeq \left[I-(1-q)x\sum_{j=1}^3 \frac{\begin{bmatrix}
				\tilde{G_j}&O\\
				\ast&H_j
			\end{bmatrix}}{x+e_j}\right]\bigg|_{e_{i}\leftrightarrow e_{i+1}}.
		\end{align}
		We define
		\begin{align}
			C=\begin{bmatrix}
				C_{11}&C_{12}\\
				C_{21}&C_{22}
			\end{bmatrix}=\begin{bmatrix}
				C_{11}^{3\times 3}&C_{12}^{3\times 6}\\
				C_{21}^{6\times 3}&C_{22}^{6\times 6}
			\end{bmatrix}
		\end{align}
		such that
		\begin{align}
			C\left[I-(1-q)x\sum_{j=1}^3\frac{\begin{bmatrix}
					{G_j'}&O\\
					\ast&H_j
			\end{bmatrix}}{x+s_i.e_j}\right]C^{-1}=\left[I-(1-q)x\sum_{j=1}^3 \frac{\begin{bmatrix}
				\tilde{G_j}&O\\
				\ast&H_j
			\end{bmatrix}}{x+e_j}\right]\bigg|_{e_{i}\leftrightarrow e_{i+1}}.
		\end{align}
		We note that $C$ is independent from $x$.
		From simple calculations, we find that $k$-th column vector $c_k$ of $C_{12}$ satisfies
		\begin{align}\label{ckhk}
			\left[I-(1-q)x\sum_{j=1}^3\frac{\tilde{G}_j}{x+e_j}\right]\bigg|_{e_{i}\leftrightarrow e_{i+1}}c_k=h_k c_k,
		\end{align}
		where
		\begin{align}
			I-(1-q)x\sum_{j=1}^3 \frac{H_j}{x+e_j}=\mathrm{diag}(h_1,h_2,h_3,h_4,h_5,h_6).
		\end{align}
		Due to the expression \eqref{H}, we have
		\begin{align}
			h_j=\begin{cases}
				\dfrac{q^\lambda x+e_j}{x+e_j}& (j=1,2,3)\\
				1&(j=4,5,6)
			\end{cases}.
		\end{align}
		We put $x=-q^{-\lambda }e_j$ ($j=1,2,3$).
		Then we have $h_j=0$.
		On the other hand we have $\displaystyle \det \left[I-(1-q)x\sum_{j=1}^3\frac{\tilde{G}_j}{x+e_j}\right]\bigg|_{e_{i}\leftrightarrow e_{i+1}}\neq 0$ since \eqref{dettildeAapp} and $e_i\neq e_j\ (i\neq j)$.
		Thus we have $c_j=0$.
		We also take $x\to\infty$ in \eqref{ckhk} with $k=4,5,6$.
		It follows from \eqref{tildeGsum} that we have $[\lambda]c_k=c_k$, in particular $c_k=0$ due to the assumption $[\lambda]=(1-(\kappa e_1e_2e_3)^{-1})/(1-q)\neq 1$.
		Therefore we have $C_{12}=O$.
		This leads us to 
		\begin{align}
			I-(1-q)x\sum_{j=1}^3\frac{{G}_j'}{x+s_i.e_j}=C_{11}\left[I-(1-q)\sum_{j=1}^3\frac{\tilde{G}_j}{x+e_j}\right]\bigg|_{e_{i}\leftrightarrow e_{i+1}}C_{11}^{-1}.
		\end{align}
		This completes the proof of \eqref{braid0i}.
		
		Next we prove \eqref{braid00}.
		As an action of parameters, we have
		\begin{align}
			s_0.e_i=\begin{cases}
				e_i&(i=1,2,3)\\
				\kappa e_1e_2e_3 e_i&(i=4,5,6,7,8,9)
			\end{cases},\quad s_0.\kappa=\frac{1}{\kappa (e_1e_2e_3)^2}.
		\end{align}
		Thus we find
		\begin{align}
			s_0.g=g,\quad s_0. \lambda =-\lambda,
		\end{align}
		where $g=(-x/e_1)_\infty(-x/e_2)_\infty(-x/e_3)_\infty$ and $q^\lambda =(\kappa e_1e_2e_3)^{-1}$.
		Therefore we have
		\begin{align}
			s_0.s_0.A=g^{-1}\circ mc_{-\lambda}\circ g\circ g^{-1}\circ mc_\lambda\circ g.A=g^{-1}\circ mc_{-\lambda}\circ mc_\lambda\circ g.A.
		\end{align}
		Using Theorem \ref{thmcompose} (ii), we find
		\begin{align}
			g^{-1}\circ mc_{-\lambda}\circ mc_\lambda\circ g.A\simeq g^{-1}\circ g=\mathrm{id}.
		\end{align}
		This finishes the proof of \eqref{braid00}.

	\end{proof}
	\begin{rem}
		We hope that $s_0.s_3.s_0.A\simeq s_3.s_0.s_3.A$, $s_3.s_3.A\simeq A$ and $s_i.s_3.A\simeq s_3.s_i.A$ ($i\neq 0$) hold by putting the $s_3$-action suitably.
		It follows from direct calculations that
		\begin{align}
			&s_3.s_0.s_3.A=(g'')^{-1}.s_3.(mc_\mu \circ g'.s_3.A),\\
			&s_0.s_3.s_0.A=(g'')^{-1}\circ mc_\nu.s_3(mc_\lambda\circ g.A),
		\end{align}
		where $g'=(-x/e_1)_\infty(-x/e_2)_\infty(-x/e_3)_\infty$, $g''=(-x/e_1)_\infty(-x/e_2)_\infty(-x/\kappa e_1e_2e_3e_4)_\infty$, $q^\mu=(\kappa e_1e_2e_4)^{-1}$ and $q^\nu=e_3/e_4$.
		We remark that $mc_\nu\circ mc_\lambda\simeq mc_\mu$ since $\nu+\lambda=\mu$.
	\end{rem}
	
	\section{Proof of Proposition \ref{propsc}}\label{appC}
	In this appendix we prove Proposition \ref{propsc}
	\begin{proof}[Proof of Proposition \ref{propsc}]
		We define
		\begin{align}
			&A=(I+x X_1)(I+xX_2)(I+xX_3)=I+A^{(1)}x+A^{(2)}x^2+\kappa Ix^3,
		\end{align}
		where $X_1$, $X_2$, $X_3$ are given by \eqref{XXtrig}, \eqref{X2trig} and \eqref{trigdetX2}, i.e.
		\begin{align}
			&X_1=\begin{bmatrix}
				1/e_1&a_1&a_2\\
				0&1/e_2&a_3\\
				0&0&1/e_3
			\end{bmatrix},\ X_3=\begin{bmatrix}
				1/e_7&0&0\\
				a_4&1/e_8&0\\
				a_5&a_6&1/e_9
			\end{bmatrix},
			\\&X_2=\kappa X_1^{-1}X_3^{-1},
			\\&\det(I+X_2)=(1+x/e_4)(1+x/e_5)(1+x/e_3).
		\end{align} 
		We put
		\begin{align}
			&A=\begin{bmatrix}
				b_1&b_2&b_3\\\ast&\ast&\ast\\\ast&\ast&\ast
			\end{bmatrix},\ 
			(T_xA)A=\begin{bmatrix}
				c_1&c_2&c_3\\\ast&\ast&\ast\\\ast&\ast&\ast
			\end{bmatrix},\ 
			(T_x^2A)(T_xA)A=\begin{bmatrix}
				d_1&d_2&d_3\\\ast&\ast&\ast\\\ast&\ast&\ast
			\end{bmatrix}.
		\end{align}
		If $y={}^\mathrm{T}(y_1,y_2,y_3)$ satisfies the equation \eqref{trigeq}, we have
		\begin{align}
			T_x y_1=b_1y_1+b_2y_2+b_3y_3,\ 
			T_x^2 y_1=c_1y_1+c_2y_2+c_3y_3,\ 
			T_x^3 y_1=d_1y_1+d_2y_2+d_3y_3.
		\end{align}
		This leads us to
		\begin{align}
			(p_3 T_x^3+p_2 T_x^2+p_1 T_x+p_0)y_1=0,
		\end{align}
		where
		\begin{align}
			&p_3=\begin{vmatrix}
				b_2&b_3\\c_2&c_3
			\end{vmatrix},\ 
			p_2=-\begin{vmatrix}
				b_2&b_3\\d_2&d_3
			\end{vmatrix},\ 
			p_1=\begin{vmatrix}
				c_2&c_3\\d_2&d_3
			\end{vmatrix},\ 
			p_0=-\begin{vmatrix}
				b_1&b_2&b_3\\c_1&c_2&c_3\\d_1&d_2&d_3
			\end{vmatrix}.
		\end{align}
		It follows from simple calculations that $p_0$ is divisible by $\det A$.
		Since $A=I+O(x)$ at $x=0$ and $A=\kappa I x^3+O(x^2)$ at $x=\infty$, there are polynomials $p_i'$ ($i=0,1,2,3$) such that
		\begin{align}
			&p_i=x^3 p_i',\\
			&\deg p_0'=12,\ \deg p_1'=9,\ \deg p_2'=6,\ \deg p_3'=3.
		\end{align}
		When we apply a gauge transformation $Y=Cy$ where
		\begin{align}
			C=\begin{bmatrix}
				1&0&0\\0&\ast&\ast\\0&\ast&\ast
			\end{bmatrix},
		\end{align}
		the equation satisfied by $Y_1$ is the same as one satisfied by $y_1$.
		By the condition $e_8\neq e_9$, there is a matrix $C\in GL(3)$ in the above form such that
		\begin{align}
			(I+(-e_8)X_3)C=\begin{bmatrix}
				\ast&\ast&0\\\ast&\ast&0\\\ast&\ast&0
			\end{bmatrix}.
		\end{align}
		Hence, we have
		\begin{align}
			&C^{-1}AC|_{x\to-e_8}=\begin{bmatrix}
				\ast&\ast&0\\\ast&\ast&0\\\ast&\ast&0
			\end{bmatrix},\\
			&C^{-1}(T_xA)AC|_{x\to-e_8}=\begin{bmatrix}
				\ast&\ast&0\\\ast&\ast&0\\\ast&\ast&0
			\end{bmatrix},\\
			&C^{-1}(T_x^2A)(T_xA)AC|_{x\to-e_8}=\begin{bmatrix}
				\ast&\ast&0\\\ast&\ast&0\\\ast&\ast&0
			\end{bmatrix}.
		\end{align}
		Therefore we find $p_i'(-e_8)=0$.
		We get $p_i'(-e_9)=0$ in a similar manner.
		We put $p_i''=p_i/(x+e_8)(x+e_9)$, then $p_i'$ ($i=0,1,2,3$) are polynomial in $x$ and
		\begin{align}
			&(p_3''T_x^3+p_2''T_x^2+p_1''T_x+p_0'')y_1=0,\\
			&\deg p_0''=10,\ \deg p_1''=7,\ \deg p_2''=4,\ \deg p_3''=1.
		\end{align}
		We define $f$ as $p_3''(f)=0$.
		We scale $p_3''$ as $x-f$.
		
		We put $x=-e_8/q$.
		Due to the determinant of $A$, $A(-e_8/q)$ is invertible.
		Since
		\begin{align}
			p_3(x)A(q^2x)A(qx)A(x)+p_2(x)A(qx)A(x)+p_1(x)A(x)+p_0(x)I=\begin{bmatrix}
				0&0&0\\\ast&\ast&\ast\\\ast&\ast&\ast
			\end{bmatrix},
		\end{align}
		we have
		\begin{align}
			p_3(-e_8/q)A(-e_8q)A(-e_8)+p_2(-e_8/q)A(-e_8)+p_1(-e_8/q)I+p_0(-e_8/q)A(-e_8/q)^{-1}=\begin{bmatrix}
				0&0&0\\\ast&\ast&\ast\\\ast&\ast&\ast
			\end{bmatrix}.
		\end{align}
		Similar to the above method, we have $p_0(-e_8/q)=0$ and $p_0(-e_9/q)=0$.
		
		We summarize the above calculations.
		We have
		\begin{align}
			&(p_3''T_x^3+p_2''T_x^2+p_1'T_x+p_0)y_1=0,\\
			&\deg p_0''=10,\ \deg p_1''=7,\ \deg p_2''=4,\ \deg p_3''=1,\\
			&p_0''\propto (x+e_1)(x+e_2)(x+e_3)(x+e_4)(x+e_5)(x+e_6)(x+e_7)(qx+e_8)(qx+e_9).
		\end{align}
		%We put $f$ as $p_3''(f)=0$.
		We apply the gauge transformation $z=((-x/e_1)_\infty(-x/e_2)_\infty(-x/e_3)_\infty)^{-1}y_1$
		Then we have
		\begin{align}
			\label{scPP}&(P_3T_x^3+P_2T_x^2+P_1T_x+P_0)z=0,\\
			&P_3=(x-f)\prod_{i=1}^3(1+qx/e_i)(1+q^2x/e_i),\\
			&P_2= p_2''\prod_{i=1}^3(1+qx/e_i),\\
			&P_1=p_1'',\\
			&P_0=p_0''\prod_{i=1}^3(1+x/e_i)^{-1}\propto(x+e_4)(x+e_5)(x+e_6)(x+e_7)(qx+e_8)(qx+e_9),\\
			&\deg P_i=7.
		\end{align}
		This clarify that the equation \eqref{scPP} has the following characteristic roots:
		\begin{itemize}
			\item $-e_4$, $-e_5$, $-e_6$, $-e_7$, $-e_8/q$, $-e_9/q$ and another root at $x=0$,
			\item $\{e_i/q,e_i/q^2\}$ (double roots) for $i=1,2,3$ and $f$.
		\end{itemize}
		The other root at $x=0$ is found in later.
		
		Next we derive characteristic roots of \eqref{scPP} at $x=0$ and $x=\infty$.
		We put
		\begin{align}
			%&A=(I+x X_1)(I+xX_2)(I+xX_3)=I+A^{(1)}x+A^{(2)}x^2+\kappa Ix^3,\\
			&B=A/(1+x/e_1)(1+x/e_2)(1+x/e_3).
		\end{align}
		We have
		\begin{align}
			T_xy'=By',\quad y'=((-x/e_1)_\infty(-x/e_2)_\infty(-x/e_3)_\infty)^{-1}y.
		\end{align}
		We define
		\begin{align}
			B_0=\frac{I-B}{x},\quad B_\infty=x({I-\kappa' B})
		\end{align}
		where $\kappa'=(\kappa e_1e_2e_3)^{-1}$.
		Then we have
		\begin{align}
			D_0y'=B_0y',\quad D_\infty y'=B_\infty.
		\end{align}
		where $D_0=\dfrac{1}{x}(1-T_x)$, $D_\infty=x(1-\kappa' T_x)$.
		From simple calculations we obtain
		\begin{align}
			&D_0^2y'=(D_0 B_0+(T_xB_0)B_0)y',\\
			&D_0^3y'=[D_0(D_0B_0+(T_xB_0)B_0)+T_x(D_0B_0+(T_xB_0)B_0)B_0]y',\\
			&D_\infty^2 y'=(D_\infty B_\infty+(T_xB_\infty)B_0 \kappa' x^2)y',\\
			&D_\infty^3y'=[D_\infty(D_\infty B_\infty+(T_xB_\infty)B_0 \kappa' x^2)+T_x(D_\infty B_\infty+(T_xB_\infty)B_0)B_0 \kappa' x^2]y'
		\end{align}
		Since 
		\begin{align}
			&B_0=B_0^{(0)}+O(x)\ (\mbox{at $x=0$}),\\
			&B_0=B_{0}^{(\infty)}x^{-2}+O(x^{-3})\ (\mbox{at $x=\infty$}),\quad B_\infty=B_\infty^{(\infty)}+O(x^{-1})\ (\mbox{at $x=\infty$}),
		\end{align}
		we have
		\begin{align}
			&D_0y'=O(1)y',\ D_0^2y'=O(1)y',\ D_0^3y'=O(1)y'\quad(\mbox{at $x=0$}),\\
			&D_\infty y'=O(1)y',\ D_\infty^2 y'=O(1)y',\ D_\infty^3 y'=O(1)y'\quad (\mbox{at $x=\infty$}).
		\end{align}
		Therefore the equation \eqref{scPP} can be rewritten as
		\begin{align}
			&(r_3^{(0)}D_0^3+r_2^{(0)}D_0^2+r_1^{(0)}D_0+r_0^{(0)})z=0,\\
			&(r_3^{(\infty)}D_\infty^3+r_2^{(\infty)}D_\infty^2+r_1^{(\infty)}D_\infty+r_0^{(\infty)})z=0,
		\end{align}
		where 
		\begin{align}
			r_i^{(0)}=O(1)\ \mbox{at $x=0$},\ r_i^{(\infty)}=O(1)\ \mbox{at $x=\infty$}.
		\end{align}
		As operators, we find
		\begin{align}
			&D_0=\frac{1}{x}(1-T_x),\ D_0^2=\frac{1}{x^2}(1-T_x)(1-T_x/q),\ D_0^3=\frac{1}{x^3}(1-T_x)(1-T_x/q)(1-T_x/q^2),\\
			&D_\infty=x(1-\kappa'T_x),\ D_\infty^2=x^2(1-\kappa'T_x)(1-\kappa' q T_x),\ D_\infty^3=x^3(1-\kappa'T_x)(1-\kappa' qT_x)(1-\kappa' q^2T_x).
		\end{align}
		This means that the equation \eqref{scPP} has triple roots $\{1,q,q^2\}$ at $x=0$ and triple roots $\{\kappa e_1e_2e_3,\kappa e_1e_2e_3 q^{-1},\kappa e_1e_2e_3 q^{-2}\}$ at $x=\infty$.
		Due to the $q$-Fuchs relation, we obtain that the last root of \eqref{scPP} at $x=0$ is $f/q$.
		This completes the proof that \eqref{scPP} has the point configuration in Proposition \ref{propsc}.
		
		We finally prove the condition \eqref{apparent}.
		Due to the above calculations, we have
		\begin{align}
			P_3(T_x^2B)(T_xB)B+P_2(T_xB)B+P_1B+P_0I=\begin{bmatrix}
				0&0&0\\
				\ast&\ast&\ast\\
				\ast&\ast&\ast
			\end{bmatrix}.
		\end{align}
		Since $P_3(f)=P_0(f/q)=0$, we have
		\begin{align}
			\label{l1}&P_2(f)B(fq)B(f)+P_1(f)B(f)+P_0(f)=\begin{bmatrix}
				0&0&0\\
				\ast&\ast&\ast\\
				\ast&\ast&\ast
			\end{bmatrix},\\
			&P_3(f/q)B(fq)B(f)B(f/q)+P_2(f/q)B(f)B(f/q)+P_1(f/q)B(f/q)=\begin{bmatrix}
				0&0&0\\
				\ast&\ast&\ast\\
				\ast&\ast&\ast
			\end{bmatrix}.
		\end{align}
		From the latter equation we have
		\begin{align}
			\label{l2}&P_3(f/q)B(fq)B(f)+P_2(f/q)B(f)+P_1(f/q)I=\begin{bmatrix}
				0&0&0\\
				\ast&\ast&\ast\\
				\ast&\ast&\ast
			\end{bmatrix}.
		\end{align}
		We focus on the first components of \eqref{l1} and \eqref{l2}.
		The tuples $(P_2(f),P_1(f),P_0(f))$ and $(P_3(f/q),P_2(f/q),P_1(f/q))$ satisfy the same system of linear equations.
		By the form of $B$, the rank of this system is generically $2$ or $3$.
		Hence, $(P_2(f),P_1(f),P_0(f))$ is proportional  to $(P_3(f/q),P_2(f/q),P_1(f/q))$.
	\end{proof}
	\begin{rem}
		The explicit expression of $f$ is given by
		\begin{align}
			f=&\frac{f_1}{f_2},\\
			f_1=&-\kappa ^2 a_1^2 a_3^3 a_6^2 e_1^2 e_3^3 e_8^3 e_9^3 e_2^3-\kappa ^3 a_1 a_2 a_3 a_6 e_1^2 e_3^3 e_8^3 e_9^3 e_2^3+\kappa ^2 a_1 a_2 a_3^2 a_6^2 e_1 e_3^3 e_8^3
			e_9^3 e_2^3
			-\kappa ^3 a_1^2 a_3 e_1^2 e_3^3 e_8^2 e_9^3 e_2^3
			\notag\\
			&
			+\kappa ^2 a_1^2 a_3^2 a_6 e_1 e_3^3 e_8^2 e_9^3 e_2^3-\kappa ^2 a_1^2 a_3^2 a_6 e_1^2 e_3^2 e_8^2 e_9^3
			e_2^3-\kappa ^3 a_1 a_2 e_1^2 e_3^2 e_8^3 e_9^2 e_2^3+2 \kappa ^2 a_1^2 a_3^2 a_6 e_1^2 e_3^2 e_8^3 e_9^2 e_2^3
			\notag\\
			&
			+2 \kappa ^2 a_1 a_2 a_3 a_6 e_1 e_3^2 e_8^3 e_9^2
			e_2^3
			-q \kappa ^2 a_1^2 a_3^2 a_6 e_1^2 e_3^3 e_8^2 e_9^2 e_2^3+\kappa ^2 a_1^2 a_3 e_1 e_3^2 e_8^2 e_9^2 e_2^3-\kappa ^2 a_1^2 a_3 e_1^2 e_3 e_8^2 e_9^2 e_2^3
			\notag\\
			&+q
			\kappa ^2 a_1^2 a_3 e_1^2 e_3^2 e_8 e_9^2 e_2^3
			+\kappa ^2 a_1^2 a_3 e_1^2 e_3 e_8^3 e_9 e_2^3+\kappa ^2 a_1 a_2 e_1 e_3 e_8^3 e_9 e_2^3-q \kappa ^2 a_1^2 a_3 e_1^2
			e_3^2 e_8^2 e_9 e_2^3
			\notag\\
			&
			-2 \kappa ^2 a_1 a_2 a_3^2 a_6^2 e_1^2 e_3^3 e_8^3 e_9^3 e_2^2
			+\kappa ^3 a_2^2 a_6 e_1^2 e_3^3 e_8^3 e_9^3 e_2^2-\kappa ^2 a_2^2 a_3 a_6^2 e_1
			e_3^3 e_8^3 e_9^3 e_2^2+\kappa ^3 a_1 a_2 e_1^2 e_3^3 e_8^2 e_9^3 e_2^2
			\notag\\
			&
			+\kappa ^2 a_1^2 a_3^2 a_6 e_1^2 e_3^3 e_8^2 e_9^3 e_2^2+2 \kappa ^2 a_1 a_2 a_3 a_6 e_1^2
			e_3^2 e_8^2 e_9^3 e_2^2+\kappa ^2 a_1^2 a_3 e_1 e_3^3 e_8 e_9^3 e_2^2-q \kappa  a_1 a_2 e_1 e_3 e_8^2 e_2^2
			\notag\\
			&
			-2 \kappa ^2 a_1 a_2 a_3 a_6 e_1^2 e_3^2 e_8^3 e_9^2
			e_2^2-\kappa ^2 a_2^2 a_6 e_1 e_3^2 e_8^3 e_9^2 e_2^2-q \kappa  a_1^2 a_3 e_1 e_3^2 e_9^2 e_2^2+2 q \kappa ^2 a_1 a_2 a_3 a_6 e_1^2 e_3^3 e_8^2 e_9^2 e_2^2
			\notag\\
			&
			+\kappa ^2
			a_1^2 a_3 e_1^2 e_3^2 e_8^2 e_9^2 e_2^2-\kappa  a_1 a_2 a_3 a_6 e_3^2 e_8^2 e_9^2 e_2^2-2 \kappa  a_1^2 a_3^2 a_6 e_1 e_3^2 e_8^2 e_9^2 e_2^2+\kappa ^2 a_1 a_2 e_1^2
			e_3 e_8^2 e_9^2 e_2^2
			\notag\\
			&
			-\kappa  a_1 a_2 a_3 a_6 e_1 e_3 e_8^2 e_9^2 e_2^2-q \kappa ^2 a_1 a_2 e_1^2 e_3^2 e_8 e_9^2 e_2^2-\kappa  a_1^2 a_3 e_3^2 e_8 e_9^2 e_2^2-q
			\kappa  a_1 a_2 a_3 a_6 e_1 e_3^2 e_8 e_9^2 e_2^2
			\notag\\
			&+\kappa  a_1^2 a_3 e_1 e_3 e_8 e_9^2 e_2^2
			+q \kappa ^2 a_1 a_2 e_1^2 e_3^2 e_8^2 e_9 e_2^2-q \kappa  a_1 a_2 a_3 a_6
			e_1 e_3^2 e_8^2 e_9 e_2^2-\kappa  a_1 a_2 e_1 e_8^2 e_9 e_2^2
			\notag\\
			&-\kappa  a_1 a_2 e_3 e_8^2 e_9 e_2^2
			-2 \kappa  a_1^2 a_3 e_1 e_3 e_8^2 e_9 e_2^2+q \kappa  a_1^2 a_3 e_1
			e_3^2 e_8 e_9 e_2^2+q \kappa  a_1 a_2 e_1 e_3 e_8 e_9 e_2^2
			\notag\\
			&
			+\kappa ^2 a_2^2 a_3 a_6^2 e_1^2 e_3^3 e_8^3 e_9^3 e_2-2 \kappa ^2 a_1 a_2 a_3 a_6 e_1^2 e_3^3 e_8^2 e_9^3
			e_2-\kappa ^2 a_2^2 a_6 e_1 e_3^3 e_8^2 e_9^3 e_2-\kappa ^2 a_2^2 a_6 e_1^2 e_3^2 e_8^2 e_9^3 e_2
			\notag\\
			&
			-\kappa ^2 a_1 a_2 e_1 e_3^3 e_8 e_9^3 e_2+q \kappa  a_1 a_2 e_1
			e_3^2 e_9^2 e_2-q \kappa ^2 a_2^2 a_6 e_1^2 e_3^3 e_8^2 e_9^2 e_2-\kappa ^2 a_1 a_2 e_1^2 e_3^2 e_8^2 e_9^2 e_2
			\notag\\
			&
			+\kappa  a_2^2 a_6 e_3^2 e_8^2 e_9^2 e_2+3 \kappa  a_1
			a_2 a_3 a_6 e_1 e_3^2 e_8^2 e_9^2 e_2+\kappa  a_2^2 a_6 e_1 e_3 e_8^2 e_9^2 e_2+\kappa  a_1 a_2 e_3^2 e_8 e_9^2 e_2
			\notag\\
			&
			-\kappa  a_1^2 a_3 e_1 e_3^2 e_8 e_9^2 e_2+q \kappa
			a_2^2 a_6 e_1 e_3^2 e_8 e_9^2 e_2-\kappa  a_1 a_2 e_1 e_3 e_8 e_9^2 e_2+q a_1 a_2 e_3 e_8 e_2
			+q \kappa  a_2^2 a_6 e_1 e_3^2 e_8^2 e_9 e_2
			\notag\\
			&
			+\kappa  a_1 a_2 e_1 e_3
			e_8^2 e_9 e_2-q a_1 a_2 e_3 e_9 e_2-q \kappa  a_1 a_2 e_1 e_3^2 e_8 e_9 e_2+a_1 a_2 e_8 e_9 e_2
			+a_1^2 a_3 e_3 e_8 e_9 e_2
			\notag\\
			&
			-q a_2^2 a_6 e_3 e_8 e_9 e_2+\kappa ^2 a_2^2
			a_6 e_1^2 e_3^3 e_8^2 e_9^3-\kappa  a_2^2 a_6 e_1 e_3^2 e_8^2 e_9^2+\kappa  a_1 a_2 e_1 e_3^2 e_8 e_9^2-a_1 a_2 e_3 e_8 e_9,\\
			f_2&=\kappa ^3 a_1^2 a_3^2 a_6 e_1^2
			e_3^3 e_8^3 e_9^3 e_2^3-\kappa ^3 a_1^2 a_3 e_1^2 e_3^2 e_8^2 e_9^3 e_2^3+\kappa ^3 a_1^2 a_3 e_1^2 e_3^2 e_8^3 e_9^2 e_2^3-q \kappa ^3 a_1^2 a_3 e_1^2 e_3^3 e_8^2
			e_9^2 e_2^3
			\notag\\
			&
			-q \kappa ^3 a_1 a_2 e_1^2 e_3^2 e_8^2 e_9^2 e_2^3+q \kappa ^2 a_1 a_2 a_3 a_6 e_1 e_3^2 e_8^2 e_9^2 e_2^3+q \kappa ^2 a_1^2 a_3 e_1 e_3^2 e_8 e_9^2
			e_2^3+q \kappa ^2 a_1 a_2 e_1 e_3 e_8^2 e_9 e_2^3
			\notag\\
			&
			-2 \kappa ^3 a_1 a_2 a_3 a_6 e_1^2 e_3^3 e_8^3 e_9^3 e_2^2+\kappa ^3 a_1 a_2 e_1^2 e_3^2 e_8^2 e_9^3 e_2^2+\kappa ^2
			a_1 a_2 a_3 a_6 e_1 e_3^2 e_8^2 e_9^3 e_2^2+\kappa ^2 a_1^2 a_3 e_1 e_3^2 e_8 e_9^3 e_2^2
			\notag\\
			&
			-q \kappa  a_1 a_2 e_1 e_8^2 e_2^2-\kappa ^3 a_1 a_2 e_1^2 e_3^2 e_8^3 e_9^2
			e_2^2+\kappa ^2 a_1 a_2 a_3 a_6 e_1 e_3^2 e_8^3 e_9^2 e_2^2-q \kappa  a_1^2 a_3 e_3^2 e_9^2 e_2^2
			\notag\\
			&
			+q \kappa ^3 a_1 a_2 e_1^2 e_3^3 e_8^2 e_9^2 e_2^2+q \kappa ^2 a_1
			a_2 a_3 a_6 e_1 e_3^3 e_8^2 e_9^2 e_2^2-q \kappa  a_2^2 a_3 a_6^2 e_3^2 e_8^2 e_9^2 e_2^2+q \kappa ^2 a_1 a_2 a_3 a_6 e_1^2 e_3^2 e_8^2 e_9^2 e_2^2
			\notag\\
			&
			-q \kappa  a_1 a_2
			a_3^2 a_6^2 e_1 e_3^2 e_8^2 e_9^2 e_2^2-\kappa ^2 a_1^2 a_3 e_1 e_3^2 e_8^2 e_9^2 e_2^2+q \kappa ^2 a_2^2 a_6 e_1 e_3^2 e_8^2 e_9^2 e_2^2-\kappa ^2 a_1 a_2 e_1 e_3
			e_8^2 e_9^2 e_2^2
			\notag\\
			&
			+q \kappa ^2 a_1^2 a_3 e_1 e_3^3 e_8 e_9^2 e_2^2+q \kappa ^2 a_1^2 a_3 e_1^2 e_3^2 e_8 e_9^2 e_2^2-2 q \kappa  a_1 a_2 a_3 a_6 e_3^2 e_8 e_9^2
			e_2^2+q \kappa ^2 a_1 a_2 e_1 e_3^2 e_8 e_9^2 e_2^2
			\notag\\
			&
			-q \kappa  a_1^2 a_3^2 a_6 e_1 e_3^2 e_8 e_9^2 e_2^2+\kappa ^2 a_1 a_2 e_1 e_3 e_8^3 e_9 e_2^2-q \kappa ^2 a_1 a_2
			e_1 e_3^2 e_8^2 e_9 e_2^2+q \kappa ^2 a_1 a_2 e_1^2 e_3 e_8^2 e_9 e_2^2
			\notag\\
			&
			-q \kappa  a_2^2 a_6 e_3 e_8^2 e_9 e_2^2-2 q \kappa  a_1 a_2 a_3 a_6 e_1 e_3 e_8^2 e_9 e_2^2-q
			\kappa  a_1 a_2 e_3 e_8 e_9 e_2^2-q \kappa  a_1^2 a_3 e_1 e_3 e_8 e_9 e_2^2
			\notag\\
			&
			+\kappa ^3 a_2^2 a_6 e_1^2 e_3^3 e_8^3 e_9^3 e_2-\kappa ^2 a_2^2 a_6 e_1 e_3^2 e_8^2 e_9^3
			e_2-\kappa ^2 a_1 a_2 e_1 e_3^2 e_8 e_9^3 e_2-\kappa ^2 a_2^2 a_6 e_1 e_3^2 e_8^3 e_9^2 e_2
			\notag\\
			&
			-q \kappa  a_1^2 a_3 e_1 e_3^2 e_9^2 e_2-q \kappa ^2 a_2^2 a_6 e_1 e_3^3
			e_8^2 e_9^2 e_2-q \kappa ^2 a_2^2 a_6 e_1^2 e_3^2 e_8^2 e_9^2 e_2+q \kappa  a_2^2 a_3 a_6^2 e_1 e_3^2 e_8^2 e_9^2 e_2
			\notag\\
			&
			+\kappa ^2 a_1 a_2 e_1 e_3^2 e_8^2 e_9^2
			e_2+\kappa  a_2^2 a_6 e_3 e_8^2 e_9^2 e_2-q \kappa ^2 a_1 a_2 e_1 e_3^3 e_8 e_9^2 e_2-q \kappa ^2 a_1 a_2 e_1^2 e_3^2 e_8 e_9^2 e_2
			\notag\\
			&
			+\kappa  a_1 a_2 e_3 e_8 e_9^2
			e_2+q a_1 a_2 e_8 e_2+q \kappa  a_2^2 a_6 e_3^2 e_8^2 e_9 e_2-\kappa  a_1 a_2 e_3 e_8^2 e_9 e_2
			+q \kappa  a_2^2 a_6 e_1 e_3 e_8^2 e_9 e_2
			\notag\\
			&+q a_1^2 a_3 e_3 e_9 e_2+q
			\kappa  a_1 a_2 e_3^2 e_8 e_9 e_2+q a_1 a_2 a_3 a_6 e_3 e_8 e_9 e_2+q \kappa  a_1 a_2 e_1 e_3^2 e_9^2+
			q \kappa  a_2^2 a_6 e_1 e_3^2 e_8 e_9^2
			\notag\\
			&
			-q a_1 a_2 e_3 e_9-q
			a_2^2 a_6 e_3 e_8 e_9
		\end{align}
		This can be checked by the direct calculation (on some software).
	\end{rem}
	\section*{Acknowledgement}
	The author would like to thank Yasuhiko Yamada for fruitful discussions and constant encouragements.
	The author is also grateful to Shunya Adachi for continuous discussions on the ($q$-)middle convolution, especially the method to introduce $mc_\lambda$ via the $q$-Okubo type equation has been constructed in private communications with him.


\begin{thebibliography}{99}
		\bibitem{AT} Y. Arai and K. Takemura, Reformulation of $q$-middle convolution and applications, arXiv:2503.11214.
		\bibitem{Boalch} P.~P. Boalch, Quivers and difference Painlev\'e{} equations, in {\it Groups and symmetries}, 25--51, CRM Proc. Lecture Notes, 47, Amer. Math. Soc., Providence, RI.
		\bibitem{DR1} M. Dettweiler and S. Reiter, An algorithm of Katz and its application to the inverse Galois problem, J. Symbolic Comput. {\bf 30} (2000), no.~6, 761--798.
		\bibitem{DR2} M. Dettweiler and S. Reiter, Middle convolution of Fuchsian systems and the construction of rigid differential systems, J. Algebra {\bf 318} (2007), no.~1, 1--24.
		\bibitem{DF} V.~S. Dotsenko and V.~A. Fateev, Conformal algebra and multipoint correlation functions in $2$D statistical models, Nuclear Phys. B {\bf 240} (1984), no.~3, 312--348.
		\bibitem{Filipuk} G.~V. Filipuk, On the middle convolution and birational symmetries of the sixth Painlev\'e{} equation, Kumamoto J. Math. {\bf 19} (2006), 15--23.
		\bibitem{FJM} B. Feigin, M. Jimbo and E. Mukhin, Remarks on $q$-difference opers arising from quantum toroidal algebras, J. Phys. A: Math. Theor. {\bf 57} (2024) 485201, 32 pp.
		\bibitem{FN} T. Fujii and T. Nobukawa, Hypergeometric solutions for variants of the $q$-hypergeometric equation, arXiv:2207.12777, to appear in Funkcial. Ekvac.
		\bibitem{HF} Y. Haraoka and G.~V. Filipuk, Middle convolution and deformation for Fuchsian systems, J. Lond. Math. Soc. (2) {\bf 76} (2007), no.~2, 438--450.
		\bibitem{Ito2020} M. Ito, $q$-difference systems for the Jackson integral of symmetric Selberg type, SIGMA Symmetry Integrability Geom. Methods Appl. {\bf 16} (2020), Paper No. 113, 31 pp.
		\bibitem{Ito2023} M. Ito, Gauss decomposition and $q$-difference equations for Jackson integrals of symmetric Selberg type, Ryukyu Math. J. {\bf 36} (2023), 1--47
		\bibitem{JS} M. Jimbo and H. Sakai, A $q$-analog of the sixth Painlev\'e{} equation, Lett. Math. Phys. {\bf 38} (1996), no.~2, 145--154.
		\bibitem{Kac} V.~G. Kac and P. Cheung, {\it Quantum calculus}, Universitext, Springer, New York, 2002.
		\bibitem{KNY} K. Kajiwara, M. Noumi and Y. Yamada, Geometric aspects of Painlev\'e equations, J. Phys. A: Math. Theor. {\bf 50} (2017), 073001, 164 pp.
		\bibitem{Katz} N.~M. Katz, {\it Rigid local systems}, Annals of Mathematics Studies, 139, Princeton Univ. Press, Princeton, NJ, 1996.
		\bibitem{Moriyama} S. Moriyama, Spectral theories and topological strings on del Pezzo geometries, J. High Energy Phys. {\bf 2020}, no.~10, 154, 53 pp.
		\bibitem{MY} S. Moriyama and Y. Yamada, Quantum representation of affine Weyl groups and associated quantum curves, SIGMA Symmetry Integrability Geom. Methods Appl. {\bf 17} (2021), Paper No. 076, 24 pp.
		\bibitem{Murata} M. Murata, Lax forms of the $q$-Painlev\'e{} equations, J. Phys. A {\bf 42} (2009), no.~11, 115201, 17 pp.
		\bibitem{NRY} M. Noumi, S.~N.~M. Ruijsenaars and Y. Yamada, The elliptic Painlev\'e{} Lax equation vs. van Diejen's 8-coupling elliptic Hamiltonian, SIGMA Symmetry Integrability Geom. Methods Appl. {\bf 16} (2020), Paper No. 063, 16 pp.
		\bibitem{Park} K. Park, A $3 \times 3$ Lax form for the $q$-Painlev\'e{} equation of type $E_6$, SIGMA Symmetry Integrability Geom. Methods Appl. {\bf 19} (2023), Paper No. 094, 17 pp.
		\bibitem{Sakai} H. Sakai, Rational surfaces associated with affine root systems and geometry of the Painlev\'e{} equations, Comm. Math. Phys. {\bf 220} (2001), no.~1, 165--229.
		\bibitem{SakaiE6} H. Sakai, Lax form of the $q$-Painlev\'e{} equation associated with the $A^{(1)}_2$ surface, J. Phys. A {\bf 39} (2006), no.~39, 12203--12210.
		\bibitem{SY} H. Sakai and M. Yamaguchi, Spectral types of linear $q$-difference equations and $q$-analog of middle convolution, Int. Math. Res. Not. IMRN {\bf 2017}, no.~7, 1975--2013.
		\bibitem{SST} S. Sasaki, S. Takagi and K. Takemura, $q$-middle convolution and $q$-Painlev\'e{} equation, SIGMA Symmetry Integrability Geom. Methods Appl. {\bf 18} (2022), Paper No. 056, 21 pp.
		\bibitem{Takemura} K. Takemura, Middle convolution and Heun's equation, SIGMA Symmetry Integrability Geom. Methods Appl. {\bf 5} (2009), Paper 040, 22 pp.
		\bibitem{Takemuradeg} K. Takemura, Degenerations of Ruijsenaars--van Diejen operator and $q$-Painlev\'e{} equations, J. Integrable Syst. {\bf 2} (2017), no.~1, xyx008, 27 pp.
		\bibitem{Yamada} Y. Yamada, Lax formalism for $q$-Painlev\'e{} equations with affine Weyl group symmetry of type $E^{(1)}_n$, Int. Math. Res. Not. IMRN {\bf 2011}, no.~17, 3823--3838.
		\bibitem{Yoshida} M. Yoshida, {\it Hypergeometric functions, my love}, Aspects of Mathematics, E32, Friedr. Vieweg, Braunschweig, 1997.
	\end{thebibliography}
\end{document}